\documentclass[10pt]{amsart}
\usepackage{amsthm,amsfonts}
\usepackage{amsmath,amssymb,amsfonts}
\newtheorem{theorem}{Theorem}[section]
\newtheorem{definition}{Definition}[section]
\newtheorem{prop}{Proposition}[section]
\newtheorem{remark}{Remark}[section]
\newtheorem{lemma}{Lemma}[section]
\newtheorem{cor}{Corollary}[section]
\numberwithin{equation}{section}
\newcommand{\w}{\omega}
\newcommand{\mK}{\mathcal K}
\newcommand{\bw}{\mbox{\boldmath $\omega$}}
\newcommand{\di}{\mbox{div }}
\newcommand{\curl}{\mbox{curl }}

\title{Global existence of a weak solution of the incompressible Euler equations
with helical
symmetry and $L^p$ vorticity}

\author[A. C. Bronzi]{A. C. Bronzi}
\address[Anne C. Bronzi]{Deptartment of Mathematics\\
Statistics and Computer Science\\
University of Illinois \\
Chicago, IL 60607, U.S.A.}
\email{annebronzi@gmail.com}

\author[M. C. Lopes Filho]{M. C. Lopes Filho}
\address[M. C. Lopes Filho]{Instituto de Matem\'atica\\
Universidade Federal do Rio de Janeiro\\
Cidade Universit\'aria -- Ilha do Fund\~ao\\
Caixa Postal 68530\\
21941-909 Rio de Janeiro, RJ -- Brasil.}
\email{mlopes@im.ufrj.br}

\author[H. J. Nussenzveig Lopes]{H. J. Nussenzveig Lopes}
\address[H. J. Nussenzveig Lopes]{Instituto de Matem\'atica\\
Universidade Federal do Rio de Janeiro\\
Cidade Universit\'aria -- Ilha do Fund\~ao\\
Caixa Postal 68530\\
21941-909 Rio de Janeiro, RJ  -- Brasil.}
\email{hlopes@im.ufrj.br}

\begin{document}
\maketitle




\vspace{0.5cm}

\pagestyle{plain}

\begin{abstract}
We prove the global existence of a helical weak solution of the 3D Euler equations,
in full space,
for an initial velocity with helical symmetry, without swirl and whose
initial vorticity is compactly supported in the axial plane and belongs to
$L^p$, for some $p>\frac{4}{3}$.
This result is an extension of the existence part of the work of B. Ettinger and
E. Titi, \cite{Titi}, who studied
well-posedness of the Euler equations with helical symmetry without swirl, with
bounded initial vorticity, in a helical pipe.
\end{abstract}

\section{Introduction}
\par Consider the 3D incompressible Euler equations,
describing the motion
of an ideal incompressible fluid in $\mathbb R^3$ with initial velocity
$u^0=u^0(x)$,
\begin{eqnarray}\label{E}
\left\{\begin{array}{l}
\partial_t u + (u\cdot\nabla)u+\nabla p=0\\
\mbox{\textit{div} }u=0\\
u(x,0)=u^0(x).
\end{array}\right.
\end{eqnarray}
Above, $u=u(t,x)\in \mathbb R^3$ is the velocity field, $p=p(t,x)\in\mathbb R$
is
the scalar pressure and $(t,x)\in \mathbb R\times \mathbb R^3$.

Another physically relevant dynamic variable for incompressible flows is $\omega
= \omega(t,x) = \mbox{ curl }u(t,x)$,
the {\it vorticity}, which satisfies the evolution equation:
\begin{equation} \label{Vor}
\partial_t \omega + (u \cdot \nabla) \omega = (\omega \cdot \nabla) u.
\end{equation}

Existence of global smooth solutions for the incompressible 3D Euler equations
with smooth initial
data is an important open problem, with a large associated literature. Local (in
time) existence with smooth
initial data is known, see \cite{Kato}. Global existence of a weak solution has
only recently been established
by E. Wiedemann in \cite{W2012} in the context of {\it wild solutions}, see also
\cite{DeSz1},
for arbitrary initial divergence-free velocity in $L^2$. The vorticity
associated with wild solutions does not, in principle, satisfy
equation \eqref{Vor}, so that another important open problem is the global
existence of a weak solution of equation
\eqref{E} whose vorticity is also a weak solution of equation \eqref{Vor};  we
will refer to such solutions as {\it tame} weak solutions.

In contrast, in dimension two, the Euler equations are globally well-posed for
smooth initial
data, see \cite{Majda}. This distinction is usually attributed to the term
$(\omega \cdot \nabla) u$,
in equation \eqref{Vor}, responsible for {\it vortex stretching}; this term
vanishes identically for 2D flows.
One way to explore the gap between 2D and 3D flows is to consider flows with
more complicated symmetries, such as axial symmetry and also helical symmetry.
There is a large literature devoted
to axisymmetric Euler flows, which naturally breaks down into two cases:
axisymmetric flows without swirl, and
the general case, with swirl. For axisymmetric flows without swirl, global
well-posedness of smooth solutions is due to A. Majda and to
X. St. Raymond, see \cite{M86,SR94,GZ07}. Global existence of tame weak
solutions was first proved by D. Chae and N. Kim
for initial vorticity in $L^p$, $p>6/5$ in \cite{CK97}. This result was later
improved to "near-vortex-sheet" data
in \cite{CI98}. For vortex sheet initial data the problem of existence remains
open, but in \cite{D92}, J.-M. Delort showed
that the argument used to obtain existence for vortex  sheet initial data in 2D
could not be extended to axisymmetric
flows without swirl, and further analysis of concentrations was performed in
\cite{JX06}. For axisymmetric flows with swirl, and for
general flows without any particular symmetry, the focus has been on criteria
and scenarios for blow-up, see \cite{C05,C08} and references therein.

Helical flows are flows invariant under a one-dimensional group of rigid motions
of Euclidean space generated by
a simultaneous rotation around an axis and a translation along the same axis.
Although as common in practice
as axisymmetric flows, helical flows have received a great deal less
mathematical attention. The literature specific to inviscid
flows with helical symmetry reduces to two articles, namely \cite{Dutrifoy},
where A. Dutrifoy
proved global well-posedness for smooth initial data and \cite{Titi}, where B.
Ettinger and E. Titi established global existence
and uniqueness of weak solutions with bounded initial vorticity; in both papers
a geometric condition, analogous to the no swirl
hypothesis for axisymmetric flows, is assumed. In general, helical flows have a
more decided 2D nature than axisymmetric flows; for example,
global existence for helical Navier-Stokes is known, see \cite{MTL90}, whereas
this same problem is open for axisymmetric
Navier-Stokes. One reason that helical flows have received less attention might
be that the helical symmetry reduction is
algebraically more complicated. The present article is part of a research
programme aimed at investigating incompressible helical flows
with a view towards both the inviscid singularity problem, and the mathematical
treatment of weak solutions. Specifically, our main
result in this article is an analogue of the $L^p$ axisymmetric existence result
of Chae and Kim; we obtain the critical
exponent $4/3$ for helical flows in the full space. Previous work on helical
flow has focused on bounded helical
domains; we choose to carry out our analysis in the full space. Although this
choice creates technical complications which
we are forced to address here, the full space case is conceptually simpler, and
can be connected in a more physically natural way
with small viscosity flows.

This article is organized as follows. In Section \ref{sec1} we fix notation,
introduce basic definitions and we state
some well-known results. In Section \ref{sec2} we develop basic tools to treat
full space helical flows.
More precisely, in Section \ref{secbiot} we derive a formula for the Green's
function for the Laplacian in
$\mathbb R^3$, periodic in the $x_3$-direction, which is used to write
explicitly the
relevant Biot-Savart law. In Section \ref{sec2.3} we introduce the definition of
weak solution and in
Section \ref{sec2.4} we state and prove a local well-posedness result for smooth
solutions to problem (\ref{E}).
Finally, in sections \ref{sec3} and \ref{sec4} we prove two versions of the main
theorem, first assuming the integral
of the third component of vorticity vanishes, which implies velocity vanishes at
infinity and greatly simplifies the
analysis, and then, second, without this hypothesis.

\section{Preliminaries}\label{sec1}
Our purpose in this section is to fix notation, introduce the basic definitions
and
to recall some known results, taken mainly from \cite{Titi}, for convenience of
the reader.

Let $\kappa$ be a positive constant and define $\Omega:=\mathbb R^2\times
(-\pi\kappa,\pi\kappa)$.
We denote by $\tilde x$ the first two components of $x\in \mathbb R^3$, i.e., if
$x=(x_1,x_2,x_3)$ then
$\tilde x=(x_1,x_2)$. Let $L_{per}^p(\Omega)$
and $H_{per}^n(\Omega)$ denote the spaces of functions $f=f(x)$, $x \in \mathbb
R^3$, which are periodic in the $x_3$-direction with period
$2\kappa \pi$ and which belong, respectively, to $L^p(\mathcal O)$ and
$H^n(\mathcal O)$, for all open sets of the form $\mathcal O=\mathbb R^2\times
I$, with $I\subset \mathbb R$ a bounded open interval. Let $V_{per}^n(\Omega)$
be the
space of all vector-valued functions in $H_{per}^n(\Omega)$ which are divergence free and let
$L_{c,per}^p(\Omega)$ denote the space of functions in $L_{per}^p(\Omega)$ which
are
compactly supported in $\tilde x$, i.e. such that, for almost all $x_3\in
\mathbb R$, the support of  $f(\cdot,x_3)$
is a compact subset of $\mathbb R^2$.

We write $R_{\theta}$ for the rotation around the $x_3$-axis by
the angle $\theta$,
 \[R_{\theta}=\left(
\begin{array}{ccc}
 \cos \theta& \sin\theta & 0 \\
-\sin \theta& \cos \theta & 0 \\
0 & 0 & 1\end{array}\right),\] and we denote by $S_{\theta}$ the superposition
of $R_{\theta}$ and a
translation along the $x_3$-direction of size $\kappa \theta$,
\[S_{\theta}(x)=R_{\theta}(x)+\left(
\begin{array}{lll}
0 \\
0 \\
\kappa \theta\end{array}\right).\]

\begin{definition}\label{helical}
A vector field $u: \mathbb R^3\rightarrow \mathbb R^3$ is said to be {\it
helical} if and only
if
\begin{equation}\label{Stheta}
u(S_{\theta}x)=R_{\theta}u(x)\mbox{ for all }\theta\in\mathbb R \mbox{ and }x\in
\mathbb R^3.
\end{equation}
\end{definition}
Observe that, if a vector field $u$ is helical, then it is periodic in the
$x_3$-direction with period $2\kappa\pi$.

\begin{definition}\label{geo}
Set
\[\xi=\xi(x)\equiv (x_2,-x_1,\kappa).\]
Given a vector field $u: \mathbb R^3\rightarrow \mathbb R^3$, we define  the
helical swirl $\eta$ as  \[\eta(x)=u(x)\cdot \xi(x) \mbox{ for all }x\in
\mathbb R^3.\] If $\eta\equiv 0$ then we say that
$u$ has vanishing helical swirl.
\end{definition}

Next, we state some properties of helical flows. The
proofs can be found in \cite{Titi}.

\begin{lemma}\label{derivative} A smooth vector field $u:\mathbb R^3\rightarrow
\mathbb R^3$ is helical if and only if $(\xi(x)\cdot
\nabla)u(x)=(u_2(x),-u_1(x),0)$.
\end{lemma}

\begin{lemma}\label{w_3}
Let $u = u(x,t)$ be a smooth solution of \eqref{E}, helical and with vanishing
helical swirl
and let $\bw = \mbox{ curl } u$. Then $\bw(x,t)=\w(x,t)\xi(x)/\kappa$,
where
$\w(x,t)=\partial_1u_2(x,t)-\partial_2u_1(x,t)$, and $\bw$ satisfies the
following
equation:

\begin{eqnarray}\label{eqlem}
\partial_t \bw+(u\cdot \nabla)\bw -\dfrac{\w}{\kappa}\mathcal R u=0,
\end{eqnarray}
where $\mathcal R=\left(
\begin{array}{rrr}
 0& 1 & 0 \\
-1 & 0 & 0 \\
0 & 0 & 0\end{array}\right)$.
\end{lemma}

\begin{remark} The third equation in system (\ref{eqlem}) is a transport
equation for $\w$, given by
\begin{eqnarray}\label{third}
\partial_t \w+(u\cdot \nabla)\w=0.
\end{eqnarray}
\end{remark}

Observe that, if  $u\in L^q(\mathbb R^3;\mathbb R^3)$, for some $1\leq q\leq
\infty$, then Definitions (\ref{helical}) and (\ref{geo}) still make sense if we ask
the equalities to hold for almost every $x\in \mathbb R^3$. Furthermore, in this
case we
have a result analogous to Lemma \ref{derivative}, which is stated below, and can
be proved in the same fashion.

\begin{lemma}\label{equiv}Let $u\in L_{loc}^q(\mathbb R^3;\mathbb R^3)$ for
some $1\leq q\leq \infty$. Then, $u$ is helical if and only if $u$ is
$2\pi\kappa$-periodic with respect to the third component, in the sense of
distributions,  and if
\begin{eqnarray}\label{form-equiv}
\int_{\Omega} (D\Psi(x)\xi(x))\cdot u(x)dx=\int_{\Omega}(\mathcal
R\Psi(x))\cdot u(x) dx,
\end{eqnarray}
for all $\Psi
\in \mathcal C^{\infty}_{c,per}(\Omega;\mathbb R^3)$.
\end{lemma}

\section{Vorticity formulation}\label{sec2}

In this section we state and prove some basic results concerning full-space
helical flows
which we will use in the proof of the existence theorems \ref{teores} and
\ref{teo2}.
We require an explicit vorticity formulation, and an appropriate Biot-Savart
law.
We begin with the following system
\begin{eqnarray}\label{green}\left\{\begin{array}{ll}
         \curl u=\bw \\
         \di u=0\\
         |u(x)|\rightarrow 0 \mbox{ as } |\tilde x|\rightarrow \infty \\
         u \mbox{ periodic in } x_3.
         \end{array}\right.
\end{eqnarray}

In order to obtain a solution for the system above we will derive an explicit
form of the Green's function for the Laplacian in $\mathbb R^3$, periodic in
the $x_3$-direction.
In addition, we introduce a definition of weak solution for the vorticity
equation \eqref{third} and, finally, we construct a smooth approximating
sequence for weak solutions of
the vorticity equation \eqref{third}.

\subsection{Biot-Savart law}\label{secbiot}
We will being by deriving the Green's
function for the Laplacian in $\mathbb R^3$ with periodic boundary conditions in
the $x_3$-direction. Using this Green's function we will write an expression for
the Biot-Savart kernel and prove
some estimates for this kernel. We will also provide a necessary and
sufficient condition for a velocity field,
associated with a given vorticity, to be helical and to have
vanishing helical swirl (see Proposition \ref{propr}).

 \begin{prop}\label{propgreen} The Green's function for the Laplacian in
$\mathbb
R^3$ with $2\pi\kappa$-periodic boundary condition in the $x_3$-direction is
given by
 \begin{equation}\label{greenfunc}
  G(x)=\dfrac{1}{2\pi\kappa^2}\sum_{n=1}^{\infty}K_0\left(\dfrac{
|\tilde
x|n}{\kappa}\right)\cos\left(\dfrac{nx_3}{\kappa}\right)-
\dfrac{1}{2\pi\kappa^2} \log |\tilde x|,
 \end{equation}
for all $x\in\mathbb R^3$, $\tilde x\neq 0$, where $K_0$ is the modified Bessel
function of the second kind and order zero.

In particular, if $\bw:\mathbb R^3\rightarrow \mathbb R^3$ belongs to
$L_{per}^2(\Omega)$ and if $\int_{\Omega}\bw(x)dx=0$, then the solution of
the Poisson system
\begin{eqnarray}\label{poisson}\left\{\begin{array}{ll}
         \Delta \Psi=\bw \\
         |\Psi(x)|\rightarrow 0 \mbox{ as } |\tilde x|\rightarrow \infty \\
         \Psi \mbox{ periodic in } x_3.
         \end{array}\right.
\end{eqnarray}
 is given by
\[\Psi(x)=-\dfrac{1}{4\pi^2}\int_{\Omega}G(x-y) \bw(y)dy, \;\mbox{ for all }x\in
\Omega. \]
\end{prop}

This result can be proved using a standard approach by means of
Fourier analysis. We omit the proof,  and we refer the reader to
\cite{Folland1}.

\begin{definition}\label{defbs} Let $G$ be the Green's function given by
\eqref{greenfunc}. Set
\begin{equation}\label{kernel-b-s}
\mathcal K = \mathcal K(x):=\dfrac{1}{4\pi^2}\nabla G(x), \mbox{ for all }
x\in (\mathbb R^2\setminus\{0\})\times
(-\pi\kappa,\pi\kappa).
\end{equation}
We refer to the function $\mathcal K$ as the Biot-Savart kernel.

\end{definition}

In the following result we obtain an estimate for the Biot-Savart kernel.

\begin{lemma}\label{lema4} The kernel $\mathcal K$, defined by
\eqref{kernel-b-s}, satisfies the following
estimate
\begin{equation} \label{needanumber}
|\mathcal K(x)|\leq C\left(\dfrac{1}{|x|^2}+\dfrac{1}{|\tilde x|}\right),
\mbox{ for all } x\in (\mathbb R^2\setminus\{0\})\times
(-\pi\kappa,\pi\kappa).
\end{equation}
\end{lemma}

\begin{proof} Fix $x\in (\mathbb R^2\setminus\{0\})\times
(-\pi\kappa,\pi\kappa)$. First, we observe that
\[|\mathcal
K(x)|\leq\dfrac{1}{8\pi^3\kappa^2}\left|\nabla\left(\sum_{n=1}^{\infty}
K_0\left(\dfrac{|\tilde x|n}{\kappa}\right)
\cos\left(\dfrac{nx_3}{\kappa}\right)\right)\right|+\dfrac{1}{8\pi^3\kappa^2}
\dfrac{1}{|\tilde x|}.\]
In order to estimate the series involving the Bessel function we use the
following expansion in \textit{Schloeminch series} (see \cite{MOS}),
\begin{equation}\label{series}
\sum_{n=1}^{\infty} K_0\left(\dfrac{|\tilde
x|}{\kappa}n\right)\cos\left(\dfrac{x_3}{\kappa}n\right)=
\dfrac{1}{2} \left(\ln\left(\dfrac{\gamma }{4\pi\kappa}\right)+\ln|\tilde
x|\right)
+\dfrac{\pi\kappa}{2|x|} +\end{equation}
\[
+\dfrac{\pi\kappa}{2}\sum_{m=1}^{\infty}\left[\dfrac{1}{
\sqrt{|\tilde x|^2+(2\pi\kappa m-x_3)^2}}-\dfrac{1}{2\pi\kappa m}\right]+\]
\[+\dfrac{\pi\kappa}{2}\sum_{m=1}^{\infty}\left[\dfrac{1}{\sqrt{|\tilde
x|^2+(2\pi\kappa m+x_3)^2}}-\dfrac{1}{2\pi\kappa m}\right].
\]

To obtain the desired estimate \eqref{needanumber} for the kernel $\mathcal K$ we need to estimate
the gradient of the series in (\ref{series}). It is not hard to see that the $m$-th
term
in each of the two series on the right-hand-side of (\ref{series}) is bounded by
a multiple of
$|x|/m(2\pi\kappa m\pm x_3)$ and, hence, both series converge pointwise.
Furthermore, for each of these two series
we have that the derivative of the $m$-th term  is bounded by $1/(2\pi\kappa
(m-1/2))^2$, so that both  series
of derivatives are uniformly convergent. Thus, we can differentiate both series
in (\ref{series}) term by term and obtain that
\[\nabla\left(\sum_{n=1}^{\infty} K_0\left(\dfrac{|\tilde
x|}{\kappa}n\right)\cos\left(\dfrac{x_3}{\kappa}n\right)\right)\]
\[=\dfrac{1}{2|\tilde x|^2}(\tilde x,0)
-\dfrac{\pi\kappa}{2|x|^3}x+\dfrac{\pi\kappa}{2}\left(\sum_{m=1}^{\infty}
f_m^{-}(x)+\sum_{m=1}^{\infty}  f_m^{+}(x)\right),\]
where
\[f_m^{\pm}(x)=\dfrac{1}{(|\tilde x|^2+(2\pi\kappa m\pm
x_3)^2)^{3/2}}(-x_1,-x_2,\mp2\pi\kappa m- x_3)).\]

Next, let us estimate the series $\sum_{m=1}^\infty f_m^-$ and
$\sum_{m=1}^\infty f_m^+$. The idea  is to
compare the series with their
corresponding integrals. In order to do so, define
$h^{\pm}(t)=\dfrac{1}{(|\tilde x|^2+(2\pi\kappa t\pm x_3)^2)^{3/2}}$ and observe
that $h^{\pm}$ is decreasing for all $t\geq 1$ and $h^{\pm}(t)\geq h^{\pm}(1)$,
for all $0\leq t\leq 1$. Thus, we can
estimate each  series by its corresponding integral as follows:
\[\sum_{m=1}^{\infty}\dfrac{1}{(|\tilde x|^2+(2\pi\kappa m\pm x_3)^2)^{3/2}}\leq\]
\[\int_{0}^{\infty}\dfrac{1}{(|\tilde x|^2+
(2\pi\kappa t\pm x_3)^2)^{3/2}}dt=\dfrac{1}{2\pi\kappa|\tilde x|^2}
\dfrac{|x|\pm x_3}{|x|}\leq \dfrac{1}{\pi\kappa|\tilde x|^2}.\]

We obtain the following estimate for the series involving each of the first two
components of $f_m^\pm$:
\begin{equation}\label{est1}
\left|\sum_{m=1}^\infty(f_m^\pm(x))_i\right|\leq
|x_i|\sum_{m=1}^{\infty}\dfrac{1}{ (|\tilde x|^2+(2\pi\kappa m\pm
x_3)^2)^{3/2}}\leq\dfrac{|x_i|}{\pi\kappa|\tilde x|^2}, \quad \mbox{ for }
i=1,2.
\end{equation}

The last component of $f_m^\pm$ is more delicate since the terms
$(f^\pm_m)_3$ do not have the same sign and do not form a decreasing sequence
in $m$.
Set $g^{\pm}(t)=\dfrac{2\pi\kappa t\pm x_3}{(|\tilde
x|^2+(2\pi\kappa t\pm x_3)^2)^{3/2}}$ and observe that $g^\pm(t)\leq 0$ if
$t\leq \mp x_3/(2\pi\kappa)$ and $g^\pm(t)\geq 0$ if
$t\geq \mp x_3/(2\pi\kappa)$. Furthermore,
\[(g^{\pm})'(t)=\dfrac{2\pi\kappa(|\tilde x|^2-2(2\pi\kappa t\pm
x_3)^2)}{(|\tilde
x|^2+(2\pi\kappa t\pm x_3)^2)^{5/2}}.\]

Thus, $(g^{\pm})'(t)\leq 0$  if and  only if $t\geq \dfrac{1}{2\pi\kappa}
\left(\dfrac{|\tilde x|}{\sqrt{2}}\mp x_3\right)$. We will divide our
analysis in two cases.  First we assume that
$\dfrac{1}{2\pi\kappa}\left(\dfrac{|\tilde x|}{\sqrt{2}}\mp
x_3\right)\geq 1$  and we split the series as follows:
\[\sum_{m=1}^{\infty}\dfrac{2\pi\kappa m\pm x_3}{(|\tilde x|^2+(2\pi\kappa m\pm
x_3)^2)^{3/2}}= \]
\[ = \sum_{m=1}^{d}\dfrac{2\pi\kappa m\pm x_3}{(|\tilde x|^2+(2\pi\kappa m\pm
x_3)^2)^{3/2}}+
\sum_{m=d+1}^{\infty}\dfrac{2\pi\kappa m\pm x_3}{(|\tilde x|^2+(2\pi\kappa m\pm
x_3)^2)^{3/2}},\]
where $d=\left[\dfrac{1}{2\pi\kappa}\left(\dfrac{|\tilde x|}{\sqrt{2}}\mp
x_3\right)\right]$ (here, $[y]$ denotes the greatest integer less than y).
Therefore we have:
\[\sum_{n=1}^{\infty}\dfrac{2\pi\kappa n\pm x_3}{(|\tilde x|^2+(2\pi\kappa n\pm
x_3)^2)^{3/2}}\leq \]
\[ \leq d\dfrac{2\pi\kappa d\pm x_3}{(|\tilde x|^2+(2\pi\kappa d\pm x_3)^2)^{3/2}}+
\int_d^{\infty}\dfrac{2\pi\kappa t\pm x_3}{(|\tilde x|^2+(2\pi\kappa t\pm
x_3)^2)^{3/2}}dt =\]
\[=d\dfrac{2\pi\kappa d\pm x_3}{(|\tilde x|^2+(2\pi\kappa d\pm x_3)^2)^{3/2}}+
\dfrac{1}{2\pi\kappa}\dfrac{1}{\sqrt{|\tilde x|^2+(2\pi\kappa d\pm
x_3)^2}}\leq\]\[\leq
 \dfrac{1}{2\pi\kappa}\left(\dfrac{|\tilde x|}{\sqrt{2}}\mp
x_3\right)\dfrac{1}{\sqrt{2}|\tilde x|^2}+\dfrac{1}{2\pi\kappa}\dfrac{1}{|\tilde
x|}\leq \] \[\leq \dfrac{1}{2\pi\kappa}\left(\dfrac{2|\tilde
x|}{\sqrt{2}}\right)\dfrac{1}{\sqrt{2}|\tilde
x|^2}+\dfrac{1}{2\pi\kappa}\dfrac{1}{|\tilde x|}=
\dfrac{1}{\pi\kappa}\dfrac{1}{|\tilde x|}.\]

For the last inequality observe that, since
$\dfrac{1}{2\pi\kappa}\left(\dfrac{|\tilde x|}{\sqrt{2}}\mp x_3\right)\geq 1$
we get  $\dfrac{|\tilde x|}{\sqrt{2}}\mp x_3\leq
\dfrac{2|\tilde x|}{\sqrt{2}}$.

Hence,
\begin{equation}
\sum_{n=1}^{\infty}\dfrac{2\pi\kappa n\pm x_3}{(|\tilde
x|^2+(2\pi\kappa n\pm
x_3)^2)^{3/2}}\leq \dfrac{1}{\pi\kappa}\dfrac{1}{|\tilde x|}.
\end{equation}

Next assume that $\dfrac{1}{2\pi\kappa}\left(\dfrac{|\tilde x|}{\sqrt{2}}\mp
x_3\right)\leq 1$.  Then it is clear that
\[\sum_{m=1}^{\infty}\dfrac{2\pi\kappa m\pm x_3}{(|\tilde x|^2+
(2\pi\kappa m\pm x_3)^2)^{3/2}}\leq\]\[\leq\dfrac{2\pi\kappa \pm x_3}{(|\tilde x|^2+
(2\pi\kappa \pm x_3)^2)^{3/2}}+\int_1^{\infty}\dfrac{2\pi\kappa t\pm
x_3}{(|\tilde x|^2+(2\pi\kappa t\pm x_3)^2)^{3/2}}dt\]
\[=\dfrac{2\pi\kappa \pm x_3}{(|\tilde x|^2+(2\pi\kappa \pm
x_3)^2)^{3/2}}+\dfrac{1}{2\pi\kappa}\dfrac{1}{\sqrt{|\tilde x|^2+
(2\pi\kappa \pm x_3)^2}}\]
\[\leq\dfrac{1}{\sqrt{|\tilde x|^2+(2\pi\kappa \pm
x_3)^2}}\left(\dfrac{1}{2\pi\kappa \pm x_3}+
\dfrac{1}{2\pi\kappa}\right)\leq \dfrac{3}{2\pi\kappa|\tilde x|}. \]

Therefore, combining the last two estimates, we obtain that
\begin{equation}\label{est2}
\sum_{m=1}^\infty (f^{\pm}_m(x))_3=\sum_{n=1}^{\infty}
\dfrac{2\pi\kappa n\pm x_3}{(|\tilde
x|^2+(2\pi\kappa n\pm
x_3)^2)^{3/2}}\leq \dfrac{3}{2\pi\kappa|\tilde x|}.
\end{equation}

Finally, combining (\ref{est1}) and (\ref{est2}) we find
\[\left|\sum_{m=1}^{\infty}  f_m^{-}(x)+\sum_{m=1}^{\infty}
f_m^{+}(x)\right|\leq 2 (|x_1|+|x_2|)\dfrac{1}{\pi\kappa}
\dfrac{1}{|\tilde x|^2}+2\dfrac{3}{2\pi\kappa}\dfrac{1}{|\tilde
x|}\leq \dfrac{5}{\pi\kappa}\dfrac{1}{|\tilde x|}.\]

Consequently,
\[|\mathcal K(x)|\leq C\left(\dfrac{1}{|x|^2}+\dfrac{1}{|\tilde x|}\right).\]

\end{proof}

Next, we provide a decay property for a convolution-type operator
associated with the Biot-Savart kernel.

\begin{lemma}\label{lemaK}Let $\Phi \in \mathcal
C_{c,per}^{\infty}(\Omega;\mathbb R^3)$ be a  vector field such that
$\int_\Omega\Phi(x)dx=0$. Let $I=I(x)$ be given by
\[I(x):=\int_{\Omega}\mathcal K(x-y)\times \Phi(y)dy, \;\mbox{ for all } x\in
\Omega.\]
 Then, $|I(x)|=\mathcal
O(|\tilde x|^{-2})$ as $|\tilde x|\rightarrow
\infty$.
\end{lemma}

\begin{proof}Let $R>0$ be such that $supp\; \Phi\subset B(0,R)\times \mathbb R$,
where $B(0,R)$ is the ball in $\mathbb R^2$ centered at the origin,  with
radius $R$. Consider any $x\in \Omega$ such that $|\tilde x|\geq 2R$.  Using
\eqref{kernel-b-s} and the fact that $K'_{0}=-K_1$ we obtain that
\[\mathcal
K(x)=\dfrac{1}{8\pi^3\kappa^2}\nabla\left(\sum_{n=1}^{\infty}K_0\left(\dfrac{
|\tilde x|n}{\kappa}\right)
\cos\left(\dfrac{nx_3}{\kappa}\right)\right)-\dfrac{1}{8\pi^3\kappa^2}\dfrac{1}{
|\tilde x|^2}(\tilde x,0)\]
\[=\dfrac{-1}{8\pi^3\kappa^2}\left(\sum_{n=1}^{\infty}\left(K_1\left(\dfrac{
|\tilde x|n}{\kappa}\right)
\cos\left(\dfrac{nx_3}{\kappa}\right)\dfrac{n\tilde x}{\kappa|\tilde
x|},K_0\left(\dfrac{|\tilde x|n}{\kappa}\right)
\sin\left(\dfrac{nx_3}{\kappa}\right)\dfrac{n}{\kappa}\right)\right)\]\[-\dfrac{1}{
8\pi^3\kappa^2}\dfrac{1}{|\tilde x|^2}(\tilde x,0)
=\mathcal K_1(x) -\mathcal K_2(x).\]

To estimate $\mathcal K_1$ we recall that Bessel function of the second kind
$K_{\nu}$ is positive and decreasing for all $\nu>0$, $\int_0^\infty
tK_0(t)dt=1$ and $\int_0^\infty tK_1(t)dt=\pi/2$. Thus,
\[|\mathcal K_1(x)|\leq
\dfrac{1}{8\pi^3\kappa^2}\sum_{n=1}^{\infty}\dfrac{n}{\kappa}
\left(K_1\left(\dfrac{|\tilde x|n}{\kappa}\right)+
K_0\left(\dfrac{|\tilde x|n}{\kappa}\right)\right) \]
\[\leq
\dfrac{1}{8\pi^3\kappa^3}\int_{0}^{\infty}t\left(K_1\left(\dfrac{|\tilde
x|t}{\kappa}\right)+
K_0\left(\dfrac{|\tilde x|t}{\kappa}\right)\right)dt=\]\[=
\dfrac{1}{8\pi^3\kappa^3}\int_{0}^{\infty}\dfrac{s\kappa}{|\tilde
x|}\left(K_1(s)+K_0(s)\right)\dfrac{\kappa}{|\tilde x|}ds\leq \dfrac{C}{|\tilde
x|^2}.\]
Thus,
\begin{eqnarray}\label{estk1}
\left|\int_{\Omega} \mathcal K_1(x-y)\times\Phi(y)dy\right|=\mathcal O(|\tilde
x|^{-2}) \mbox{ as }|\tilde x|\rightarrow \infty.
\end{eqnarray}

Next, we observe that, for $|\tilde x|\geq 2R$ and $|\tilde y|\leq R$, we
have that
\[|\tilde x -\tilde y|^{-2}=|\tilde x|^{-2}+\mathcal O(|\tilde x|^{-3}),\]
so that
\[\int_{\Omega}\mathcal K_2(x-y)\times
\Phi(y)dy=\dfrac{1}{8\pi^3\kappa^2}\int_{\Omega}\dfrac{1}{|\tilde x -\tilde
y|^2}
(\tilde x-\tilde y,0)\times \Phi(y)dy\]
\[=\dfrac{1}{8\pi^3\kappa^2}\dfrac{1}{|\tilde x|^2}(\tilde x,0)\times
\int_{\Omega}\Phi(y)dy+\mathcal O(|\tilde x|^{-2})=\mathcal O(|\tilde
x|^{-2}).\]

In view of the estimates obtained we find: $|I(x)|=\mathcal O(|\tilde x|^{-2})$
for $|\tilde x|\geq 2R$.
\end{proof}

\begin{cor}\label{corK}
Let $\Phi \in \mathcal
C_{c,per}^{\infty}(\Omega;\mathbb R^3)$ and assume that
$\int_{\Omega}\Phi(x)dx=0$. Then the function $I$ given by
\[I=I(x):=\int_{\Omega}\mathcal K(x-y)\times \Phi(y)dy\]
belongs to $L^r(\Omega;\mathbb R^3)$, for any $r>1$.
\end{cor}

\begin{proof}
Let $R>0$ be such that $supp\; \Phi\subset
B(0,R)\times \mathbb R$. It follows from the proof of Lemma \ref{lemaK} that
$|I(x)|=\mathcal O(|\tilde
x|^{-2})$ for $|\tilde x|\geq 2R$. Thus, for all $r>1$, $u\in L^r(\Omega\cap
(B(0,2R)^c\times \mathbb R))$. On the other hand, using that $\mathcal K\in
L_{loc}^s(\Omega)$, for $1\leq s< 3/2$, $\Phi\in L^p(\Omega)$,
for $1\leq p\leq \infty$, and the Generalized Young inequality on $U=\Omega
\cap (B(0,3R)\times
\mathbb R)$, we obtain that
\[\|\mathcal K\ast \Phi\|_{L^r(U)}\leq \|\mathcal
K\|_{L^s(U)}\|\Phi\|_{L^p(U)},\]
where $1/r=1/p+1/s-1$. We conclude by observing that
we can cover all
$r$ in the interval $(1,\infty]$ by choosing, for instance, $s=1$ and $p=r$.
\end{proof}

We are now ready to provide a characterization of the velocity field in terms
of  vorticity, i.e.,  the so called Biot-Savart law.

\begin{prop}\label{resolvebs} Let $\Phi \in \mathcal
C_{c,per}^{\infty}(\Omega;\mathbb R^3)$ be a vector field such that $\di
\Phi=0$
and $\displaystyle \int_{\Omega}\Phi dx=0$. Then there exists a unique smooth
solution of
\begin{eqnarray}\label{biotsavart}\left\{\begin{array}{ll}
\mbox{curl } u= \Phi\\
\mbox{div } u=0\\
|u(x)|\rightarrow 0 \mbox{ as } |\tilde x|\rightarrow \infty\\
u \mbox{ periodic in } x_3
\end{array}\right.\end{eqnarray}
and it is given by
\begin{equation}\label{ex-b-s}
u=u(x):=\int_{\Omega}\mathcal K(x-y)\times
\Phi(y)dy.
\end{equation}
\end{prop}

\begin{proof}
We start by proving that the vector field given by \eqref{ex-b-s} is indeed a
solution of \eqref{biotsavart}. Since $\Phi \in \mathcal
C_{c,per}^{\infty}(\Omega;\mathbb R^3)$, it is clear that $u\in \mathcal
C_{per}^{\infty}(\Omega;\mathbb R^3)$, i.e., $u$ is smooth and periodic in
$x_3$. Moreover, using Lemma \ref{lemaK}, we have that
$u(x)=O(|\tilde x|^{-2})$, and, hence, $|u(x)|\rightarrow 0$ as $|\tilde
x|\rightarrow \infty$. Next, recall that $\mK(x)=1/(4\pi^2)\nabla G(x)$, for all
$x\in \Omega$, $\tilde x\neq 0$, where $G$ is given by \eqref{greenfunc}. Since
$\Phi \in \mathcal
C_{c,per}^{\infty}(\Omega;\mathbb R^3)$ a straightforward calculation yields
\[u(x)=-\curl\left(-\frac{1}{4\pi^2}\int_{\Omega}G(x-y)\Phi(y)dy\right).\]
Hence, $\di u=0$. Define $\Psi(x)=-1/(4\pi^2)\int_{\Omega}G(x-y)\Phi(y)dy$, for
all $x\in \Omega$, and observe that, by Proposition \ref{propgreen}, $\Delta
\Psi=\Phi$. Thus, using that $-\curl \curl \Psi=\Delta \Psi-\nabla
\di \Psi$, we obtain
\begin{equation}\label{iden1}
\curl u=\Phi-\nabla \di \Psi.
\end{equation}
To conclude the proof that $u$ is a solution of
\eqref{biotsavart} it remains only to show that
$\nabla \di \Psi=0$. To see this, introduce the vector fields $f:=\nabla
\mbox{div } \Psi$ and $g:=\curl u$. Since $\Psi,\Phi\in \mathcal
C_{per}^{\infty}(\Omega;\mathbb R^3)$, we obtain the following identities
\[f(x)=-\int_{\Omega}\mathcal
K(x-y) \mbox{div
}\Phi(y)dy\] \[g(x)=\int_{\Omega}\mathcal K(x-y)\times \mbox{curl
}\Phi(y)dy,\]
so that, by Lemma \ref{lemaK}, we have that $f=\mathcal O(|\tilde x|^{-2})$
and
$g=\mathcal O(|\tilde x|^{-2})$. Therefore, $f$ and $g$ belong to $L^2(\Omega)$.
Taking the $L^2$-inner product between $f$ and identity (\ref{iden1}) we
obtain
that
\[(f,f)_2=(\Phi,f)_2-(g,f)_2.\]

We have that $f$ is the gradient of div $\Psi$, div $g=0$ and, by hypothesis,
div $\Phi=0$. Thus, we have that $(\Phi,f)_2=(g,f)_2=0$.
Therefore, $(f,f)_2=0$, which, in turn, implies that $\nabla
\mbox{div } \Psi\equiv 0$.

Finally, let us prove that the system \eqref{biotsavart} has a unique solution.
Suppose that we have two solutions of (\ref{biotsavart}),
namely $u_1$ and $u_2$. Of course,  $U:=u_1-u_2$ is a solution
of (\ref{biotsavart}), with $\Phi=0$, that is,
\begin{eqnarray}\left\{\begin{array}{ll}\label{uni}
\mbox{curl } U= 0\\
\mbox{div } U=0\\
|U(x)|\rightarrow 0 \mbox{ as } |\tilde x|\rightarrow \infty\\
U \mbox{ periodic in } x_3.
\end{array}\right.\end{eqnarray}
Using the identity $\curl \curl U= -\Delta U+\nabla
\di U$ and the fact that $U$ is a solution of \eqref{uni} it follows that
$\Delta U=0$. Observe that $U$ can be regarded as a harmonic function in
the full-space, which is $2\pi\kappa$-periodic with respect to the
third component, and which decays to zero as $|\tilde x|$ goes to
infinity, so that $U$ is bounded in $\mathbb R^3$.
It follows from Liouville's theorem (see e.g. \cite{evans})
that $U\equiv 0$.
\end{proof}


\subsection{Weak solution}\label{sec2.3}
We begin by giving the definition of a weak solution; this formulation is
inspired on Schochet's weak
vorticity formulation for the $2D$ incompressible Euler equations, see
\cite{schochet}.

\begin{definition}\label{w-s} Let
$\w^0\in L_{c,per}^p(\Omega;\mathbb R)$ for some
$p>4/3$. We say that a scalar function $\w=\w(t,x)$ is
a weak solution of \begin{eqnarray}\label{vort}\partial_t \w+(u\cdot
\nabla)\w=0\end{eqnarray}  in the time interval $(0,T)$ with initial data
$\w^0$, if
\begin{itemize}
 \item [(i)] $\w\in L^{\infty}(0,T; L_{per}^p(\Omega;\mathbb R))$,
 \item [(ii)] the associated velocity field \[u(t,x)=\int_{\Omega}\mathcal
K(x-y)\times
\frac{\xi(y)}{\kappa} \w(t,y)dy\] belongs to $L^{\infty}(0,T;
L_{per}^2(\Omega;\mathbb R^3))$,
 \item [(iii)] for all test functions $\psi \in \mathcal
C^{\infty}_{c,per}([0,T)\times\Omega;\mathbb R)$ the following identity holds
true:
\end{itemize}
\begin{equation}\label{eqtw}
\int_0^{T}\int_{\Omega} \psi_t(t,x) \w(t,x)dxdt+\int_0^{T}
\int_{\Omega}\int_{\Omega} \mathcal
H_{\psi}(t,x,y)\w(t,y)\w(t,x)dydxdt \end{equation}
\[
+\int_{\Omega} \psi(x,0)\w^0(x)dx=0,\]
where
\[\mathcal H_{\psi}(t,x,y)=\dfrac{1}{2\kappa}\mathcal K(x-y)\cdot
\left(\xi(y)\times(\nabla\psi(t,x)-\nabla\psi(t,y))-(\xi(x)-\xi(y))\times
\nabla \psi(t,y)\right).\]
\end{definition}

\begin{remark}
We note in passing that the integrals in Definition \ref{w-s} are convergent.
Indeed, using H\"older's inequality we have that
\[\quad \left|\int_0^{T}\int_{\Omega} \psi_t(t,x) \w(t,x)dxdt+ \int_{\Omega}
\psi(x,0)\w^0(x)dx\right| \]
\[\leq \|\psi_t\|_{L^{1}(L^q)}\|\w\|_{L^{\infty}(L^p)}+\|\psi(0,\cdot)\|_{L^q}
\|\w^0\|_{L^p}<\infty,\]
where $1/p+1/q=1$.
It remains to prove that the following integral is bounded:
\[\int_0^{T} \int_{\Omega}\int_{\Omega} \mathcal
H_{\psi}(t,x,y)\w(t,y)\w(t,x)dydxdt.\]
To see that this integral is finite we divide the domain $\Omega \times \Omega
\times (0,T)$ into
subdomains as follows. Fix $\delta > 0$ and $R>0$ and let
$R_1=\{(x,y) \in \Omega\times \Omega:  |x-y|< \delta\}$,
$R_2=\{(x,y) \in \Omega\times \Omega:  |x-y|\geq \delta, |x|< R\}$
and
$R_3=\{(x,y) \in \Omega\times \Omega:  |x-y|\geq \delta, |x|\geq R\}$.
Performing estimates in each region yields the desired result; the proof follows
the same steps as in the proof of Theorem \ref{teores}, so we omit the details
here.


\end{remark}

Our next proposition consists in basic properties relating  the velocity field
and the corresponding
vorticity.

\begin{prop}\label{propr}Let $\bw\in L_{c,per}^p(\Omega;\mathbb R^3)$ and let
$u(x)=(\mathcal K\ast \bw)(x)$, for all $x\in\Omega$. Then, $u$
is helical and has vanishing helical swirl if and only if
$\bw=(\partial_1u_2-\partial_2u_1)\xi/\kappa$
in $\mathcal D'(\Omega;\mathbb R^3)$.
\end{prop}
\begin{proof}
First observe that, since $\bw=$curl $u$ in $\mathcal D'$, for all $\Psi \in
\mathcal C^{\infty}_c(\Omega;\mathbb R^3)$ we have that
$\langle\bw,\Psi\rangle=  \langle \mbox{curl }u
,\Psi\rangle=\langle u ,\mbox{curl }\Psi\rangle$. Suppose that $u$ is helical
and has vanishing swirl, that is, $\eta=u\cdot \xi=0$, then
$u_3(x)=(-x_2u_1(x)+x_1u_2(x))/\kappa$. Using this fact and Lemma
\ref{equiv} we obtain the following
\[\langle
\bw,\Psi\rangle= \dfrac{1}{\kappa}\left \langle \left(\begin{array}{cc}\kappa
u_1\\\kappa u_2\\-x_2u_1+x_1u_2\end{array}\right),\curl
\Psi\right\rangle=\]\[=\dfrac{1}{\kappa}\left \langle
\curl \left(\begin{array}{cc}\kappa
u_1\\\kappa u_2\\-x_2u_1+x_1u_2\end{array}\right), \Psi\right\rangle\]
\[=\dfrac{1}{\kappa}\langle \w\xi,\Psi\rangle+\dfrac{1}{\kappa}\left \langle
 \left(\begin{array}{cc}-\xi\cdot \nabla u_2-u_1\\\xi\cdot \nabla
u_1-u_2\\0\end{array}\right), \Psi\right\rangle\]\[=\dfrac{1}{\kappa}\langle
\w\xi,\Psi\rangle+\dfrac{1}{\kappa}\left \langle
 \left(\begin{array}{cc}u_1\\u_2\\u_3\end{array}\right),
(\xi\cdot \nabla)\left(\begin{array}{cc}-\Psi_2\\\Psi_1\\0\end{array}
\right)\right\rangle-\dfrac{1}{\kappa}\left
\langle
 \left(\begin{array}{cc}u_1\\u_2\\u_3\end{array}\right),
\mathcal R\left(\begin{array}{cc}-\Psi_2\\\Psi_1\\0\end{array}
\right)\right\rangle=\]\[=\dfrac{1}{\kappa}\langle
\w\xi,\Psi\rangle.\]

Therefore, for all $\Psi \in \mathcal C^{\infty}_c(\Omega;\mathbb R^3)$ we have
that $\displaystyle\langle\bw,\Psi\rangle=
\left\langle\w\xi/\kappa,\Psi\right\rangle.$

Conversely,  suppose that  $\bw=\w\xi/\kappa$, where
$\w=\partial_1u_2-\partial_2u_1$. First
observe that $(\mbox{curl }\bw)\cdot\xi=-2\w$. Indeed, for all
$\varphi \in \mathcal C^\infty_c(\Omega;\mathbb R)$ we have that
\[\langle (\mbox{curl }\bw)\cdot\xi,\varphi\rangle=
\langle \mbox{curl }\bw,\xi\varphi\rangle= \langle \bw,\mbox{curl
}(\xi\varphi)\rangle=  \langle \bw,\varphi\mbox{ curl
}\xi+\nabla\varphi\times\xi\rangle\]
\[= \langle \bw,-2\varphi e_3+\nabla\varphi\times\xi\rangle=
\left\langle
\dfrac{1}{\kappa}\w\xi,-2\varphi e_3+\nabla\varphi\times\xi\right\rangle=\]\[=
\left\langle
\dfrac{1}{\kappa}\w,\xi\cdot(-2\varphi
e_3+\nabla\varphi\times\xi)\right\rangle= \langle
-2\w,\varphi\rangle.\]

Since $\bw=$ curl $u$ and div $u=0$ in $\mathcal D'$ we also have that
$\mbox{curl } \bw=-\Delta u \mbox{ in } \mathcal D'$.  Furthermore,
since $\Delta (u\cdot \xi)=\Delta u\cdot \xi +u\cdot \Delta \xi
+2\sum_{i=1}^3\nabla u_i\cdot\nabla\xi_i$   and
$\Delta \xi=0$ we obtain that
\[\Delta (u\cdot \xi)= -\curl \bw \cdot
\xi+2\sum_{i=1}^3\nabla u_i\cdot\nabla\xi_i=2\w-2\w=0\]
in $\mathcal D'$. Thus, using Lemma \ref{lemaK}, $u\cdot \xi$ solves the
following equation
in the sense of distributions:
\[\left\{\begin{array}{ll}
\Delta (u\cdot \xi)=0 \;\mbox{ in }\; \Omega\\
|u(x)\cdot\xi(x)|\rightarrow 0 \mbox{ as }|\tilde x|\rightarrow \infty
         \end{array}\right.\]
Therefore, $u\cdot \xi=0$, i.e., $u$  has vanishing helical swirl.

Finally we prove that $u$ is helical. To do so, we note that, in view of Lemma \ref{equiv}, it is enough to  establish that, for all
$\Psi \in \mathcal C^{\infty}_c(\Omega;\mathbb R^3)$,
\begin{eqnarray}\label{eq6.1}
\int_{\Omega} u(x)\cdot ((\xi(x)\cdot \nabla)\Psi(x))dx=\int_{\Omega}
u(x)\cdot (\mathcal
R\Psi(x)) dx.
\end{eqnarray}

To prove \eqref{eq6.1} we begin by showing that, for all $\Psi \in \mathcal C^{\infty}_c(\Omega;\mathbb
R^3)$,
\begin{eqnarray}\label{exp=0}
\int_{\Omega}u_3(x)(\xi(x)\cdot \nabla)\Psi_3(x)dx=0.
\end{eqnarray}

Indeed, since $\curl u=\bw=\w \xi/\kappa$, we have
\[\langle u, \curl((x_1,x_2,0)\Psi_3)\rangle = \langle \curl u,
(x_1,x_2,0)\Psi_3 \rangle = \langle \bw\cdot (x_1,x_2,0), \Psi_3 \rangle = 0.\]
As $u\cdot \xi=0$ a.e., we also have that
\[\langle u, \curl((x_1,x_2,0)\Psi_3)\rangle = \left \langle u,
\left(\begin{array}{cc}
-x_2\partial_3\Psi_3\\x_1\partial_3\Psi_3\\x_2\partial_1\Psi_3-x_1\partial_2
\Psi_3 \end{array}\right)\right\rangle \]
\[= \langle -x_2u_1+x_1u_2,
\partial_3\Psi_3 \rangle +\langle u_3,x_2\partial_1\Psi_3-x_1\partial_2
\Psi_3 \rangle \]
\[ = \langle \kappa u_3,\partial_3\Psi_3 \rangle +\langle
u_3,x_2\partial_1\Psi_3-x_1\partial_2 \Psi_3 \rangle = \langle u_3, (\xi\cdot
\nabla)\Psi_3\rangle.\]
Therefore, $\langle u_3, (\xi\cdot \nabla)\Psi_3\rangle=0$, which proves
(\ref{exp=0}).

To conclude the proof that $u$ is helical it remains to prove that the right-hand-side of \eqref{eq6.1}
also vanishes. To this end it is enough to show that
\begin{eqnarray}\label{eq6.2}
\langle u_1, (\xi\cdot \nabla \Psi_1)\rangle +\langle u_2, (\xi\cdot \nabla
\Psi_2)\rangle=\langle u_1,\Psi_2\rangle +\langle u_2,-\Psi_1\rangle.
\end{eqnarray}
Let $\Phi\in \mathcal C^{\infty}_c(\Omega;\mathbb R^3)$ and
use that $\bw=(\partial_1u_2-\partial_2u_1)\xi/\kappa$ in
$\mathcal D'$ to obtain
\begin{equation}\label{eq6.3}
 \langle \bw, \Phi\rangle=\frac{1}{\kappa}\langle
(\partial_1u_2-\partial_2u_1)\xi, \Phi\rangle=-\dfrac{1}{\kappa}\left(\langle
u_2, \xi\cdot \partial_1\Phi -\Phi_2\rangle-\langle
u_1, \xi\cdot \partial_2\Phi +\Phi_1\rangle\right) \end{equation}
\[
=-\frac{1}{\kappa}\left(-\langle u_1, \Phi_1\rangle-\langle u_2,
\Phi_2\rangle-\langle u_1,\xi\cdot \partial_2\Phi \rangle+\langle u_2,
\xi\cdot \partial_1\Phi\rangle \right).
\]

Next, we observe that since $\bw=\curl u$ and $u\cdot \xi=0$ in $\mathcal D'$,
we also get that
\begin{eqnarray}\label{eq6.4}
\langle \bw, \Phi\rangle=\langle u,
\curl \Phi\rangle=\frac{1}{\kappa}\langle (\kappa u_1, \kappa u_2,
-x_2u_1+x_1u_2), \curl \Phi\rangle.
\end{eqnarray}

Now, combining (\ref{eq6.3}) and (\ref{eq6.4})
we obtain that
\[\langle u_1, (\xi\cdot \nabla (-\Phi_2))\rangle +\langle u_2, (\xi\cdot
\nabla
\Phi_1)\rangle=\langle u_1,\Phi_1\rangle +\langle u_2,\Phi_2\rangle. \]
Choosing $\Phi=(\Psi_2,-\Psi_1, 0)$ in the last expression
yields (\ref{eq6.2}). Finally, combining (\ref{eq6.2}) and (\ref{exp=0}) we
obtain (\ref{eq6.1}).
\end{proof}

\subsection{Smooth solutions}\label{sec2.4}
Our next results consists in local well-posedness of the 3D Euler equations with
helical symmetry and vanishing helical swirl  in the full space for smooth  initial data.

Recall the notation used in \cite{Majda} where $V^k(\Omega)$ stands for the set of all vector-valued functions in the Sobolev space $H^k(\Omega)$ which
are divergence free.

\begin{theorem}\label{Yo}
For any initial velocity $u^0\in V^k(\Omega)$, with $k\geq 3$, there exists
$T_0=T_0(\|u^0\|_{H^k})>0$ such that, for any $T<T_0$ there exists
a unique solution $u\in\mathcal C([0,T];
V^k(\Omega))\cap \mathcal C^1([0,T];V^{k-1}(\Omega))$ of the 3D Euler
equations. Moreover,  if  $u^0$ is helical and has vanishing helical swirl, then
$u(\cdot,t)$ also is helical and has vanishing helical swirl, for any fixed time $t \in [0,T]$.
\end{theorem}

The first part of the assertion of the previous theorem can be proved in the
same fashion as the classical result in $\mathbb R^3$. The second part can be
prove using  Dutrifoy's argument (see \cite{Dutrifoy}). We omit the proof.

Next,  we will construct a sequence of smooth solutions, using Theorem \ref{Yo} above, which
is uniformly bounded in a suitable function space.

\begin{theorem}\label{app-seq} Let $\w^0 \in L^p_{c,per}(\Omega;\mathbb  R)$,
for some $1< p\leq \infty$. For each $n\in \mathbb N$, define
$\w^0_n:=\rho_n\ast \w^0$, where $\rho_n$ is the standard mollifier in $\mathbb
R^3$ which is periodic with respect to the third component. Then,
\begin{itemize}
 \item [(i)] for each  $n\in \mathbb N$, there exists a smooth solution $\w_n\in
L^{\infty}(0,\infty; L^p_{c,per}(\Omega;\mathbb R))$
of (\ref{vort}) with initial data $\w^0_n$;
 \item [(ii)]$\{\w_n\}_n$ is  uniformly bounded in $L^{\infty}(0,\infty;$ $
L_{per}^q(\Omega;\mathbb R))$, for all $1\leq q\leq p$;
 \item [(iii)]$\{\w_n\}_n$ possesses a subsequence which converges
weak-$\ast$ in $L^{\infty}(0,\infty; L_{per}^q(\Omega; \mathbb R))$ to a limit
$\w$,  for all $1<q\leq p$;
 \item [(iv)] $\|\w(t,\cdot)\|_{L^1(\Omega)}\leq C$, for some constant
$C>0$.
\end{itemize}
\end{theorem}

\begin{proof}
For each $n\in \mathbb N$ define $u^0_n(x)=\int_{\Omega} \mathcal K(x-y)\times
(1/\kappa)\xi(y)\w^0_n(y)dy$. Then
$u^0_n\in V^k$ (see Corollary \ref{corK}) for all $k\in\mathbb N$, so we can
use Theorem \ref{Yo} to obtain, for each $n$, a time $T^n>0$
 such that, for any $T<T^n$ there exists a unique solution $u_n\in\mathcal
C^1([0,T];V^{k-1})$ of the Euler equations. Moreover, for each $n$,
$u_n$ is helical and has vanishing helical swirl. Thus, for each $n$, $\bw_n=$
curl $u_n$ solves (\ref{eqlem}) and, therefore $\w_n$ solves (\ref{third}), whence
$\w_n(t,X_n(\alpha,t))=\w^0_n(\alpha)$, where $X_n(\alpha,t)$ is
the particle trajectory, starting at $\alpha \in \Omega$, such that
\[\left\{\begin{array}{ll}
\dfrac{dX_n(t,\alpha)}{dt}=u_n(t,X_n(t,\alpha))\\
X_n(\alpha,t)\left|_{t=0}\right.=\alpha.
          \end{array}\right.\]

It follows that
\[\int_{\Omega}|\w_n(t,x)|^pdx=\int_{\Omega}|\w_n^0(\alpha)|^pd\alpha,
\mbox{ for all } 1\leq p<\infty,\;\mbox{ and
}\;\|\w_n(t,\cdot)\|_{L^\infty(\Omega)}=
\|\w^0_n\|_{L^\infty(\Omega)}.\]

Note that, since $\w_0\in L^p_{c,per}(\Omega)$, we have $\w_0\in
L_{per}^q(\Omega)$ for all $1\leq q\leq p$. Therefore,
\[\|\w_n(t,\cdot)\|_{L^q(\Omega)}=\|\w_n^0\|_{L^q(\Omega)}\leq
\|\w^0\|_{L^q(\Omega)}\leq C, \mbox{ for all } 1\leq q\leq p,\]
which implies that  $\{\w_n\}_n$ is  uniformly bounded in
$L^{\infty}(0,\infty; L_{per}^q(\Omega;\mathbb R))$, for all $ 1\leq q\leq p$.

Observe also that $\|\w^0_n\|_{L^\infty(\Omega)}\leq
C(n)\|\w^{0}\|_{L^p(\Omega)}$, hence $\|\w_n(t,\cdot)\|_{L^\infty(\Omega)}\leq
C(n)$. Now, since  $\bw_n(t,x)=\w_n(t,x)\xi(x)/\kappa$ we obtain that
\[\|\bw_n(t,\cdot)\|_{L^\infty}=\frac{1}{\kappa}\sup_{\alpha \in \;supp\;
w_n^0}| \w_n^0(\alpha)\xi(X_n(t,\alpha))|\leq C\|\w_n^0\|_{L^\infty}
\left(\kappa+\sup_{\alpha \in\; supp \;\w_n^0}|X_n(t,\alpha)|\right).\]

It remains to estimate $\sup_{\alpha \in\; supp \;\w_n^0}|X_n(t,\alpha)|$. To do
so, define $\displaystyle R_n(t)= \sup_{ \alpha \in K_n}|X_n(t,\alpha)|$ and
$K_n=supp \;\w^0_n$. Observe that for all $\alpha'\in K_n$ we have
\[\dfrac{d}{dt}|X_n(t,\alpha')|\leq \sup_{ \alpha
\in K_n}\dfrac{d}{dt}|X_n(t,\alpha)|\leq \sup_{ \alpha
\in K_n}\left|\dfrac{d}{dt}X_n(t,\alpha)\right| \]
\[\leq \sup_{\alpha \in K_n}\int_{|X_n(t,\alpha)-y|\leq 2R_n(t)}
|\mathcal K(X_n(t,\alpha)-y)||\xi(y)||\w_n(y)|dy \]
\[\leq \sup_{\alpha\in K_n}\int_{|X_n(t,\alpha)-y|\leq
4\kappa\pi}\left(\dfrac{1}{|X_n(t,\alpha)-y)|^2}+\dfrac{1}{|\tilde
X_n(t,\alpha)-\tilde y)|}\right)(|\tilde y-\tilde
X_n(t,\alpha)|+|\xi(X_n(t,\alpha))|)|\w_n(y)|dy\]
\[+C \sup_{\alpha \in K_n}\int_{4\kappa\pi\leq|X_n(t,\alpha)-y|\leq
2R_n(t)}|\mathcal \xi(y)||\w_n(y)|dy \]
\[\leq\|\w_n\|_{L^\infty} \sup_{\alpha \in K_n}\int_{|X_n(t,\alpha)-y|\leq
4\kappa\pi}\left(\dfrac{1}{|X_n(t,\alpha)-y)|}+1+\dfrac{\xi(X_n(t,\alpha))}{|
X_n(t,\alpha)-y)|^2}+\dfrac{\xi(X_n(t,\alpha))}{|\tilde X_n(t,\alpha)-\tilde
y)|}\right)dy \]
\[+CR_n(t)\|\w_n\|_{L^1}\leq C(\|\w_n\|_{L^\infty}+\|\w_n\|_{L^1})(R_n(t)+1)\leq
C(n)(R_n(t)+1),\]
where we used Lemma \ref{lema4} to estimate the kernel $\mathcal K$.

Hence,
\[\dfrac{d}{dt}|X_n(t,\alpha')|\leq C(n)(R_n(t)+1),\]
which implies, $\displaystyle|X_n(t,\alpha')|\leq
|X_n(0,\alpha')|+C(n)\int_{0}^{t}(R_n(s)+1)ds$.

Taking the supremum  in $\alpha' \in K_n$ of the last identity
we obtain
\[R_n(t)\leq C+C(n)\int_{0}^{t}(R_n(s)+1)ds.\]

By Gronwall's lemma we have that $R_n(t)\leq (C+C(n)t)e^{C(n)t}$. Thus,
$\|\bw_n(t,\cdot)\|_{L^\infty}\leq (C+C(n)t)\|\w^n_0\|_{L^\infty}e^{C(n)t}$.

We have, hence,
\[\int_{0}^t\|\bw_n(s,\cdot)\|_{L^\infty}ds\leq
\|\w^n_0\|_{L^\infty}\int_{0}^t(C+C(n)s)e^{C(n)s}ds\leq
\|\w^n_0\|_{L^\infty}(Ct+C(n)t^2)e^{C(n)t}.\]

Therefore, by the Beale-Kato-Majda theorem we conclude that, for each $n\in
\mathbb N$, $u_n$ can be continued, within its class, for all $T>0$. Therefore,
$\w_n$ is a global smooth solution of (\ref{vort}), for each $n\in\mathbb N$.

Since $\{\w_n\}$ is uniformly bounded in $L^{\infty}(0,\infty;
L^q(\Omega))$, for all $1\leq q\leq p$, then there exists a subsequence, that
we still denote by $\{\w_n\}$, that converges
weak-star in $L^{\infty}(0,\infty; L^q(\Omega))$ to a limit $\w$, for all
$1<q\leq p$.
Since the total variation of a measure is weak-star lower semicontinuous  and as $\|\w_n(t,\cdot)\|_{L^1(\Omega)}\leq C$,
we find that $\|\w(t,\cdot)\|_{L^1(\Omega)}\leq C$.
\end{proof}


\section{First existence theorem: balanced vorticity}\label{sec3}
In this section we prove existence of a weak solution to the helical
vorticity equation for initial data with vanishing integral. This condition is
needed since we are going to make use of the  Biot-Savart law \eqref{ex-b-s},
which requires this hypothesis.
\begin{theorem}\label{teores}
Let $\w_0\in L_{c,per}^p(\Omega;\mathbb R)$, for some $p>4/3$, such
that $\int_{\Omega}\w_0(x)dx=0$. Then, there exists a weak solution $\w=\w(t,x)$
of (\ref{vort}), in the sense of Definition \ref{w-s}, with initial data $\w_0$.
\end{theorem}

\begin{proof}We use Theorem \ref{app-seq} to obtain a sequence of smooth
solutions $\{\w_n\}$ of the vorticity equation satisfying the following
properties
\begin{itemize}
\item [(i)] $\w_n\in  L^\infty(0,\infty;L_{c,per}^p(\Omega;\mathbb R))$, for all
$n$;
\item [(ii)] $\{\w_n\}_{n=1}^{\infty}$ is uniformly bounded in
$L^{\infty}(0,\infty, L_{per}^q(\Omega;\mathbb R))$, for all $1\leq q\leq p$;
\item [(iii)] the velocity field $u_n(x):=\int_{\Omega}\mathcal
K(x-y)\times (\xi(y)/\kappa) \w_n(t,y)dy$ is helical and has vanishing
helical swirl for all $n$, and $\{u_n\}_{n=1}^{\infty}$ is uniformly bounded in
$L^{\infty}(0,\infty;L_{per}^2(\Omega;\mathbb R^3))$,
\item [(iv)] there exists a subsequence,  still denoted
$\{\w_n\}$, which converges
weak-star in $L^{\infty}(0,\infty; $ $L_{per}^q(\Omega; \mathbb R))$, for all $1<q\leq p$, to a limit
$\w$.
\end{itemize}

We claim that $\w$ is a weak solution of the
vorticity equation. To see this, we observe that $\w_n$ satisfies (\ref{eqtw}),
for all $n\in \mathbb N$, and, up to a subsequence,
$\w_n\stackrel{*}{\rightharpoonup}\w$ in $L^{\infty}(0,\infty;
L_{per}^q(\Omega; \mathbb R))$. Clearly, the linear terms of (\ref{eqtw})
converge, as $n\rightarrow \infty$. Hence, to obtain that $\w$
satisfies \eqref{eqtw} it is enough to prove that the nonlinear term of
(\ref{eqtw}) converges, as $n\rightarrow \infty$, and that the following
identity holds true
\begin{eqnarray*}
\lim_{n\rightarrow \infty}
\int_{0}^{\infty}\int_{\Omega}\int_{\Omega}\mathcal
H_{\psi}(t,x,y)\w_n(t,y)\w_n(t,x)dydxdt=
\int_{0}^{\infty}\int_{\Omega}\int_{\Omega}\mathcal
H_{\psi}(t,x,y)\w(t,y)\w(t,x)dydxdt.
\end{eqnarray*}

The remainder of the proof consists in showing that we can pass to the limit
in the nonlinear term, in order to establish this identity. Let $\psi \in \mathcal
C^{\infty}_{c,per}([0,\infty)\times\Omega;\mathbb R)$ and fix $\rho>0$ such that
$\mbox{supp }\psi\subset C_{\rho}$, where $C_{\rho}=\{x\in\mathbb R^3 \colon
|\tilde x|\leq \rho \mbox{ and } |x_3|\leq \pi\kappa\}$. Define
\[\alpha_n:=\dfrac{1}{16\pi^3\kappa^3}\int_{C_{\rho}}
\w_n(t,x)\partial_3\psi(x)dx \mbox{ and }
\alpha:=\dfrac{1}{16\pi^3\kappa^3}\int_{C_{\rho}}
\w(t,x)\partial_3\psi(x)dx,\]
and observe that $\lim_{n\rightarrow \infty}\alpha_n=\alpha$.

Fix $0<\delta \ll 1$ and $R\gg 2\max\{\rho,2\pi\kappa\}$ and let
$\varphi_{\delta}:
[0,\infty)\rightarrow [0,1]$ and $\zeta_R: \Omega\rightarrow [0,1]$ be smooth cutoff functions, so that

\begin{eqnarray}\label{varphi}\begin{array}{rr} \varphi_{\delta}(r)\equiv 1, & r\leq \delta,\\ \\
                            0\leq \varphi_{\delta}(r) \leq 1, & \delta<r\leq 2\delta\\ \\
                             \varphi_{\delta}(r) \equiv0, & r>2\delta;
                            \end{array}
                            \end{eqnarray}
\begin{eqnarray}\label{zeta}\begin{array}{rr}
                            \zeta_{R}(x) \equiv  1, & |x|\leq R,\\ \\
                            0 \leq \zeta_{R}(x)  \leq 1, &  R<|x|\leq 2R,\\ \\
                             \zeta_{R}(x)\equiv 0, & |x|>2R.
                            \end{array}\end{eqnarray}

Using this notation we have the following decomposition:

\[\int_{\Omega}\int_{\Omega}\mathcal H_{\psi}(t,x,y)\w_n(t,y)\w_n(t,x)dydx=
\int_{\Omega}\int_{\Omega}\mathcal
H_{\psi}(t,x,y)\varphi_{\delta}(|x-y|)\w_n(t,y)\w_n(t,x)dydx\]
\[+ \int_{\Omega}\int_{\Omega}\mathcal H_{\psi}(t,x,y)
(1-\varphi_{\delta}(|x-y|))\w_n(t,y)\w_n(t,x)dydx\equiv I^n_1+I^n_2.\]

Using the definition of $\mathcal H_{\psi}$ and the Generalized Young
Inequality with $1/s+1/p=1+1/p'$, we obtain that
\[|I^n_1|\leq \int_{\Omega}\int_{\Omega}|\mathcal H_{\psi}(t,x,y)|
\varphi_{\delta}(|x-y|)|\w_n(t,y)||\w_n(t,x)|dydx\]
\[\leq C\int_{\Omega}\int_{\Omega}\left(\dfrac{1}{|x-y|}+
\dfrac{1}{|\tilde x-\tilde y|}+1\right)
\varphi_{\delta}(|x-y|)|\w_n(t,y)||\w_n(t,x)|dydx\]
\[\leq C\left\|\left(\dfrac{1}{|\cdot|}+
\dfrac{1}{|\tilde \cdot|}+1\right)
\varphi_{\delta}(|\cdot|)\right\|_{L^s}\|\w_n(t,\cdot)\|_p\|\w_n(t,.)\|_p.\]

If $s<2$ then we find that
\[\left\|\left(\dfrac{1}{|\cdot|}+ \dfrac{1}{|\tilde \cdot|}+1\right)
\varphi_{\delta}(|\cdot|)\right\|^s_{L^s}\leq C\delta^{2-s}.\]

Recall that
$\|\w_n(t,\cdot)\|_p\leq C$. Recall, also, that we assumed that $p>4/3$, so that $p'<4$. Now, since
$1/s+1/p=1+1/p'$ we have that $s=p'/2$. Therefore we have, indeed, $s<2$ and, hence, $I^n_1\leq C \delta^{(2-s)/s} = o(1)$ as $\delta \to 0$.

Next we estimate $I^n_2$. Since $\int_\Omega\w_n(t,x)dx=0$, we can rewrite
$I^n_2$ as
\[I^n_2= \int_{\Omega}\left(\int_{\Omega}\mathcal H_{\psi}(t,x,y)
(1-\varphi_{\delta}(|x-y|))\w_n(t,y)dy+\alpha_n\right)\w_n(t,x)dx.\]

Now, we decompose $I^n_2$ as
\[I^n_2= \int_{\Omega}\left(\int_{\Omega}\mathcal H_{\psi}
(t,x,y)(1-\varphi_{\delta}(|x-y|))\w_n(t,y)dy+
\alpha_n\right)\zeta_R(x)\w_n(t,x)dx\]
\[+\int_{\Omega}\left(\int_{\Omega}\mathcal H_{\psi}(t,x,y)
(1-\varphi_{\delta}(|x-y|))\w_n(t,y)dy+\alpha_n\right)(1-\zeta_R(x))
\w_n(t,x)dx\equiv I^n_{21}+I^n_{22}.\]

We start by estimating $I^n_{21}$. We decompose the integral as
\[I^n_{21}= \int_{\Omega}\left(\int_{\Omega}\mathcal H_{\psi}
(t,x,y)(1-\varphi_{\delta}(|x-y|))\zeta_{R}
(y)\w_n(t,y)dy\right)\zeta_R(x)\w_n(t,x)dx\]
\[+\int_{\Omega}\left(\int_{\Omega}\mathcal H_{\psi}
(t,x,y)(1-\varphi_{\delta}(|x-y|))(1-\zeta_{R}(y))\w_n(t,y)dy+
\alpha_n\right)\zeta_R(x)\w_n(t,x)dx\equiv J_1^n+J_2^n.\]

We begin by analyzing $J_1^n$. Observe that the function $\mathcal
H_{\psi}(t,x,y)(1-\varphi_{\delta}(|x-y|))\zeta_{R} (y)\zeta_R(x)\in
C^\infty_c([0,\infty)\times\Omega\times \Omega)$. Therefore we
can pass to the limit as $n\rightarrow \infty$ and obtain that
\[\lim_{n\rightarrow \infty}J_1^n=\int_{\Omega}\int_{\Omega}\mathcal H_{\psi}
(t,x,y)(1-\varphi_{\delta}(|x-y|))\zeta_{R}(y)\zeta_R(x)\w(t,y)\w(t,x)dydx.\]

Next, we consider $J_2^n$. Observe
that the integrand in the inner integral, with respect to $y$, vanishes for $|y|\leq
R$. Also, for $y\in
C_{\rho}^c = \mathbb{R}^3 \setminus C_{\rho}$, we have that $\mathcal H_{\psi}(t,x,y)=1/(2\kappa)\mathcal
K(x-y)\cdot(\xi(y)\times\nabla\psi(t,x))$. Moreover, we note that $\nabla
\psi(x)$ vanishes for all $x\in C_{\rho}^c$. Therefore, it is enough to estimate
$\mathcal H_\psi$ for $|y|>R$ and $x\in C_{\rho}$. In this case, $|x-y|\geq
|y|-|x|\geq |y|/2$ and $|\tilde x-\tilde y|\geq
|x-y|-2\pi\kappa\geq |y|/4$. Thus, using the notation $\mathcal K_1$ and
$\mathcal K_2$, introduced in the proof of Lemma \ref{lemaK}, we have
\[\mathcal H_{\psi}(t,x,y)= \dfrac{1}{2\kappa}\mathcal K_1(x-y)\cdot
(\xi(y)\times\nabla\psi(t,x))-\dfrac{1}{2\kappa}\mathcal K_2(x-y)\cdot
(\xi(y)\times\nabla\psi(t,x)).\]
Introduce $G_1:=\frac{1}{2\kappa}\mathcal K_1(x-y)\cdot
(\xi(y)\times\nabla\psi(t,x))$ and recall the definition of $\mathcal K_2$ to
get that
\[\mathcal H_{\psi}(t,x,y)= G_1-\dfrac{1}{16\pi^3\kappa^3}\dfrac{(\tilde
x-\tilde y,0)}{|\tilde x-\tilde y|^2}\cdot (\xi(y)\times\nabla\psi(t,x))\]
\[=G_1-\dfrac{1}{16\pi^3\kappa^3}\dfrac{1}{|\tilde x-\tilde y|^2}\left((\tilde
x-\tilde y,0)\cdot (-\tilde y,0)\partial_3\psi(t,x)+\kappa(\tilde x-\tilde
y,0)\cdot (-\partial_2\psi(t,x),\partial_1\psi(t,x),0)\right)\]
\[=G_1-\dfrac{1}{16\pi^3\kappa^3}\dfrac{1}{|\tilde x-\tilde y|^2}\left(|\tilde
x-\tilde y|^2\partial_3\psi(t,x)+(\tilde x-\tilde y,0)\cdot (-\tilde
x,0)\partial_3\psi(t,x)\right.\]
\[\left.+\kappa(\tilde x-\tilde y,0)\cdot
(-\partial_2\psi(t,x),\partial_1\psi(t,x),0)\right).\]

Let $G_2:=-\dfrac{1}{16\pi^3\kappa^3|\tilde x- \tilde y|^2}(-(\tilde x-\tilde
y,0)\cdot (\tilde x,0)\partial_3\psi(t,x)+\kappa(\tilde x-\tilde
y,0)\cdot(-\partial_2\psi(t,x),\partial_1\psi(t,x),0))$. With this notation we now have
\[\mathcal H_{\psi}(t,x,y)=G_1-\dfrac{1}{
16\pi^3\kappa^3}\partial_3\psi(t,x)+G_2.\]

Observe that
\begin{eqnarray}\label{est3}
|G_1|\leq C\dfrac{1}{|\tilde
x-\tilde y|^2}|y|\leq C\dfrac{1}{|y|} \mbox{ and } |G_2|\leq
C\dfrac{1}{|y|}.
\end{eqnarray}

We rewrite $J_2^n$ as
\[J_2^n=\int_{\Omega}\left(\int_{\Omega}(G_1+G_2)\zeta_R(x)\w_n(t,
x)dx\right)(1-\zeta_{R}(y))\w_n(t,y)dy\]
\[-\int_{\Omega}\alpha_n(1-\zeta_{R
} (y))\w_n(t, y)dy+\alpha_n\int_{\Omega}
\zeta_R(x)\w_n(t,x)dx.\]
Using (\ref{est3}) we obtain the following estimate
\[|J_2^n|\leq \dfrac{C}{R}\|\w_n(t,.)\|_1^2+2\left|\alpha_n\int_ {\Omega}
\zeta_R(x)\w_n(t,x)dx\right|\leq \dfrac{C}{R}+C\left|\int_ {\Omega}
\zeta_R(x)\w_n(t,x)dx\right|.\]

It is clear that
\[\lim_{n\rightarrow\infty} \int_ {\Omega}\zeta_R(x)\w_n(t,x)dx=\int_{\Omega}
\zeta_R(x)\w(t,x)dx.\]
Moreover, using the Dominated Convergence theorem  we obtain  that
\[\lim_{R\rightarrow \infty}\int_{\Omega}\zeta_R(x)\w(t,x)dx=\int_
{\Omega}\w(t,x)dx=0.\]
Thus, $\lim_{R\rightarrow \infty}\lim_{n\rightarrow \infty} |J_2^n|=0$.

Next we treat $I^n_{22}$. In order to estimate $I^n_{22}$, we observe that, since  $R\gg 2\rho$, it follows that
$\mathcal H_{\psi}(t,x,y)=0$ for all $|x|>R$ and $y\in
C_{\rho}^c$. Furthermore, for $|x|>R$ and $y \in C_{\rho}$ we have that
$|x-y|\geq |x|/2$ and $|\tilde x-\tilde y|=|x-y|-2\pi\kappa\geq
|x|/4$. Hence, $1-\varphi_\delta(|x-y|)=1$ for all  $|x|>R$ and $y \in
C_{\rho}$. In this
region we can rewrite $\mathcal H_{\psi}$ as
\[\mathcal H_{\psi}(t,x,y)=\dfrac{1}{2\kappa}(\mathcal K_1(x-y)-\mathcal
K_2(x-y))\cdot \left(\xi(y)\times(-\nabla\psi(t,y))-(\xi(x)-\xi(y))\times
\nabla \psi(t,y)\right)\]
\[=\dfrac{1}{2\kappa}(\mathcal K_1(x-y)\cdot
\left(-\xi(y)\times\nabla\psi(t,y)-(\xi(x)-\xi(y))\times  \nabla
\psi(t,y)\right)+\mathcal K_2(x-y)\cdot (\xi(y)\times\nabla\psi(t,y)))\]
\[+\dfrac{1}{2\kappa}\mathcal K_2(x-y)\cdot( (\xi(x)-\xi(y))\times
\nabla \psi(t,y))\equiv F_1+F_2.\]

Observe that, for $|x|>R$ and $y\in C_{\rho}$, we get
\[|F_1|\leq C\dfrac{1}{|\tilde x-\tilde y|^2}(|y|+|\tilde
x-\tilde y|)+C\dfrac{1}{|\tilde x-\tilde y|}|y|\leq
C\left(\dfrac{1}{R^2}+\dfrac{1}{R}\right),\]
while $F_2=-\dfrac{1}{16\pi^3\kappa^3}
\partial_3 \psi(t,y),$  so that $\int_{C_\rho} F_2\w_n(t,y)dy=-\alpha_n$. Thus,
\[|I^n_{22}|\leq \int_{\Omega}\int_{C_{\rho}}|F_1|
|\w_n(t,y)||(1-\zeta_R(x))||\w_n(t,x)|dydx\]
\[+\left|\int_{\Omega}\left(\int_{C_{\rho}}F_2
\w_n(t,y)dy+\alpha_n\right)(1-\zeta_R(x))\w_n(t,x)dx\right|\]
\[\leq
\int_{\Omega}\int_{C_{\rho}}C\left(\dfrac{1}{R^2}+\dfrac{1}{R}\right)
|\w_n(t,y)||(1-\zeta_R(x))||\w_n(t,x)|dydx\]
\[\leq C\left(\dfrac{1}{R^2}+\dfrac{1}{R}
\right)\|\w_n(t,\cdot)\|^2_1\leq C\left(\dfrac{1}{R^2}+\dfrac{1}{R}
\right) \leq \dfrac{C}{R},\]
if $R>>1$.

We can now pass to the limit as $n\rightarrow \infty$ and use all
the estimates obtained to get
\[\lim_{n\rightarrow \infty}\int_{\Omega}\int_{\Omega}\mathcal
H_{\psi}(t,x,y)\w_n(t,y)\w_n(t,x)dydx=\lim_{n\rightarrow\infty}
(I^n_1+J_1^n+J_2^n+I^n_{22})\]
\[=C\delta^{(2-s)/s}+\lim_{n\rightarrow\infty}\int_{\Omega}\int_{\Omega}\mathcal H_{\psi}
(t,x,y)(1-\varphi_{\delta}(|x-y|))\zeta_{R}
(y)\zeta_R(x)\w_n(t,y)\w_n(t,x)dydx\]
\[+\left(\dfrac{C}{R}+C\lim_{n\rightarrow\infty}\left|\int_{\Omega}\w_n(t,
x)\zeta_R(x)dx\right|\right)+\dfrac{C}{R}\]
\[=\int_{\Omega}\int_{\Omega}\mathcal H_{\psi}
(t,x,y)(1-\varphi_{\delta}(|x-y|))\zeta_{R}
(y)\zeta_R(x)\w(t,y)\w(t,x)dydx\]
\[+C\delta^{(2-s)/s} +\dfrac{C}{R}+C\left|\int_{\Omega}\w(t,x)\zeta_R(x)dx\right|.\]

Now, let us pass to the limit as $R\rightarrow \infty$ and $\delta\rightarrow
0$ to obtain that
\[\lim_{n\rightarrow \infty}\int_{\Omega}\int_{\Omega}\mathcal
H_{\psi}(t,x,y)\w_n(t,y)\w_n(t,x)dydx\]
\[=\lim_{R\rightarrow \infty,\delta\rightarrow
0}\left(\int_{\Omega}\int_{\Omega}\mathcal H_{\psi}
(t,x,y)(1-\varphi_{\delta}(|x-y|))\zeta_{R}
(y)\zeta_R(x)\w(t,y)\w(t,x)dydx\right)\]
\[+\lim_{R\rightarrow \infty,\delta\rightarrow 0}\left(C\delta^{(2-s)/s}+\dfrac{C}{R}+C\left|\int_{\Omega}\w(t,x)\zeta_R(x)dx\right|\right)\]
\[=\lim_{R\rightarrow \infty,\delta\rightarrow
0}\left(\int_{\Omega}\int_{\Omega}\mathcal H_{\psi}
(t,x,y)(1-\varphi_{\delta}(|x-y|))\zeta_{R}
(y)\zeta_R(x)\w(t,y)\w(t,x)dydx\right).\]

On the other hand we have that
\[\int_{\Omega}\int_{\Omega}\mathcal H_{\psi}
(t,x,y)\w(t,y)\w(t,x)dydx=
\int_{\Omega}\int_{\Omega}\mathcal H_{\psi}
(t,x,y)\varphi_{\delta}(|x-y|)\w(t,y)\w(t,x)dydx\]
\[+\int_{\Omega}\int_{\Omega}\mathcal H_{\psi}
(t,x,y)(1-\varphi_{\delta}(|x-y|))\zeta_{R}
(y)\zeta_R(x)\w(t,y)\w(t,x)dydx\]
\[+\int_{\Omega}\left(\int_{\Omega}\mathcal H_{\psi}
(t,x,y)(1-\varphi_{\delta}(|x-y|))(1-\zeta_{R}
(y))\w(t,y)dy+\alpha\right)\zeta_R(x)\w(t,x)dx\]
\[+\int_{\Omega}\left(\int_{\Omega}\mathcal H_{\psi}
(t,x,y)(1-\varphi_{\delta}(|x-y|))\w(t,y)dy+\alpha\right)(1-\zeta_R(x))\w(t,
x)dx\]
using the same idea from the estimates obtained for the sequence $\w_n$,
\[=C\delta^{(2-s)/s}+\int_{\Omega}\int_{\Omega}\mathcal H_{\psi}
(t,x,y)(1-\varphi_{\delta}(|x-y|))\zeta_{R}
(y)\zeta_R(x)\w(t,y)\w(t,x)dydx+\]
\[+\left(\dfrac{C}{R}+C\left|\int_{\Omega}\w(t,x)\zeta_R(x)dx\right|\right)+\dfrac{C}{R}.\]

Therefore,
\[\int_{\Omega}\int_{\Omega}\mathcal H_{\psi}
(t,x,y)\w(t,y)\w(t,x)dydx\]
\[=\lim_{R\rightarrow \infty, \delta\rightarrow
0}\int_{\Omega}\int_{\Omega}\mathcal H_{\psi}
(t,x,y)(1-\varphi_{\delta}(|x-y|))\zeta_{R}
(y)\zeta_R(x)\w(t,y)\w(t,x)dydx.\]

Finally, we conclude that
\[\lim_{n\rightarrow \infty}\int_{\Omega}\int_{\Omega}\mathcal
H_{\psi}(t,x,y)\w_n(t,y)\w_n(t,x)dydx=\int_{\Omega}\int_{\Omega}\mathcal
H_{\psi}
(t,x,y)\w(t,y)\w(t,x)dydx.\]

It is clear from the {\it a priori} estimates that $\{u_n\}$ is uniformly bounded in
$L^{\infty}(0,\infty;L_{per}^2(\Omega;\mathbb R^3))$ so that there exists a subsequence, which we still denote
by $\{u_n\}$, that converges to $u\in L^{\infty}(0,\infty;L_{per}^2(\Omega;\mathbb R^3))$. It remains to prove 
that the limit $u$ is the velocity associated with the vorticity $\omega$, that is, $
u(t,x)=\int_{\Omega}\mathcal K(x-y)\times(\xi(y)/\kappa)\w(t,y) dy$. 

Fix $R>0$ and set 
$U_R=\{x\in \Omega: |\tilde x|\leq R\}$. Let us fix $\delta>0$ such that
$0<\delta\ll R$ and let $\varphi_{\delta}$ be as defined  in (\ref{varphi}), so that 
\[\|u_n(t,\cdot)-u(t,\cdot)\|_{L^2(U_R)}^2=\int_{U_R}\left|\int_{\Omega}
\mathcal K(x-y)\times
\dfrac{\xi(y)}{\kappa}(\w_n(t,y)-\w(t,y))dy\right|^2dx\]
\[\leq 2\int_{U_R}\left|\int_{\Omega}\varphi_{\delta}(|x-y|)\mathcal
K(x-y)\times
\dfrac{\xi(y)}{\kappa}(\w_n(t,y)-\w(t,y))dy\right|^2dx\]\[+2\int_{U_R}
\left|\int_{\Omega}(1-\varphi_{\delta}(|x-y|))\mathcal K(x-y)\times
\dfrac{\xi(y)}{\kappa}(\w_n(t,y)-\w(t,y))dy\right|^2dx\equiv \mathcal
I^n_1+\mathcal I^n_2.\]

We start by estimating the term $\mathcal I^n_1$. Observe that for $|x|\leq R$
and $|x-y|<2\delta$ it holds that $|\xi(y)|\leq |x-y|+|x|\leq
2\delta+R\leq 2R$ and hence 
\[\mathcal I_1^n\leq
C(R)\int_{U_R}\left(\int_{\Omega}|\varphi_{\delta}(|x-y|)||\mathcal K(x-y)|
|\w_n(t,y)-\w(t,y)|dy\right)^2dx\]\[\leq C(R)\|\varphi_{\delta}\mathcal
K\|^2_{L^q}\|\w_n-\w\|^2_{L^p},\]
where in the last inequality we used the Generalized Young inequality with
$ 1+1/2=1/q+1/p$. Since $\|\varphi_{\delta}\mathcal K\|_{L^q}\leq C\delta^{3-2q}$ and $\|\w_n-\w\|_{L^p}\leq \|\w_n\|_{L^p}+\|\w\|_{L^p}\leq C$
then $\mathcal I_1^n\leq C \delta^{3-2q}$, for all $n\in\mathbb N$. It is clear that the condition $p>4/3$ implies $q<3/2$.
Therefore, 
\[\lim_{\delta\rightarrow 0}\lim_{n\rightarrow \infty}\mathcal I_1^n=  0.\]

Now, we will estimate $\mathcal I_2^n$. Recall the notation introduced in Lemma
\ref{lemaK}, $\mathcal K=\mathcal K_1-\mathcal K_2$, so
\[\mathcal I_2^n\leq
4\int_{U_R}\left|\int_{\Omega}(1-\varphi_{\delta}(|x-y|))\mathcal K_1(x-y)\times
\dfrac{\xi(y)}{\kappa}(\w_n(t,y)-\w(t,y))dy\right|^2dx\]
\[+4\int_{U_R}\left|\int_{\Omega}(1-\varphi_{\delta}(|x-y|))\mathcal K_2(x-y)
\times \dfrac{\xi(y)}{\kappa}(\w_n(t,y)-\w(t,y))dy\right|^2dx\equiv \mathcal
F^1_n + \mathcal F^2_n.\]

We will see that $\mathcal F^1_n$ and  $\mathcal  F^2_n$ converge to
zero when $n\rightarrow \infty$. We start with $\mathcal F^1_n$. Observe that
for $|x-y|>\delta$ we have
\[|\mathcal K_1(x-y)\times \xi(y)|= |\mathcal K_1(x-y)\times
(\xi(y)-\xi(x)+\xi(x))|\] \[\leq C\dfrac{1}{|\tilde x-\tilde y|^2}(|\tilde
y-\tilde x|+|\xi(x)|) \leq C(R)\left(\dfrac{1}{|\tilde x-\tilde
y|}+\dfrac{1}{|\tilde x-\tilde y|^2}\right).\]
Thus, for every $x\in U_R$ fixed,
$(1-\varphi_{\delta}(|x-\cdot|))\mathcal K_1(x-\cdot)\times \xi(\cdot)\in
L^{r}(\Omega)$, for all $r>2$. Moreover, since $\{\w_n\}$ is
uniformly bounded in $L^{\infty}(0,\infty;L^q(\Omega))$ for all $1\leq q\leq
p$ we obtain that $\w_n(t,\cdot)\stackrel{*}{\rightharpoonup} \w(t,\cdot)$ in
$L^q(\Omega)$,
for all $1<q\leq p$, so that we can choose $q$ in such a way that $1<q\leq p$ and
$q'>2$. Therefore, we obtain that $\lim_{n\rightarrow \infty}\int_\Omega
(1-\varphi_{\delta}(|x-y|))\mathcal K_1(x-y)\times
(\xi(y)/\kappa)(\w_n(t,y)-\w(t,y))dy=0$ for every $x\in U_R$.
Moreover, we have that
\[\left|\int_\Omega (1-\varphi_{\delta}(|x-y|))\mathcal K_1(x-y)\times
\frac{\xi(y)}{\kappa}(\w_n(t,y)-\w(t,y))dy\right|\]\[\leq
C\|(1-\varphi_{\delta}(|x-\cdot|))\mathcal K_1(x-\cdot)\times
\xi\|_{L^{q'}}\|\w_n(t,\cdot)-\w(t,\cdot)\|_{L^q}\leq C_\delta,\]
where $1/q+1/q'=1$, $q\leq p$ and $q'>2$, so that the last
sequence is uniformly bounded in $n$ by a constant which is an integrable
function in $U_R$. Thus, we can apply
the Dominated Convergence theorem to obtain that $\lim_{n\rightarrow
\infty}\mathcal F^1_n=0$.

Now we will prove the convergence of the term $\mathcal F^2_n$. Fix $M>0$,
$M\gg R$ and consider $\zeta_M$ as defined in (\ref{zeta}). Let us split the
integral in the following way:
\[\int_{\Omega}(1-\varphi_{\delta}(|x-y|))\zeta_{M}(y)
\mathcal K_2(x-y)\times \dfrac{\xi(y)}{\kappa}(\w_n(t,y)-\w(t,y))dy\]
\[+\int_{\Omega}(1-\varphi_{\delta}(|x-y|))(1-\zeta_M(y))\mathcal
K_2(x-y)\times
\dfrac{\xi(y)}{\kappa}(\w_n(t,y)-\w(t,y))dy\equiv \mathcal J^1_n(t,x)+\mathcal
J^2_n(t,x).\]

In order to prove the convergence of $\mathcal J^1_n$, first we recall that
$\|\w_n(t,\cdot)\|_{L^1}\leq C$ implies that
$\lim_{n\rightarrow\infty}\int_\Omega \w_n(t,y)f(y)dy=
\int_{\Omega}\w(t,y)f(y)dy$ for all $f\in C_0(\Omega)$. Then, observe that
$f_x(y):=(1-\varphi_{\delta}(|x-y|))\zeta_{M}(y)\mathcal K_2(x-y)\times
(\xi(y)/\kappa) \in \mathcal C_0(\Omega)$ for all $x\in U_R$, thus
\[\mathcal J^1_n(t,x)=\int_{\Omega}f_x(y)(\w_n(t,y)-\w(t,y))dy\rightarrow 0
\mbox{ as }n\rightarrow \infty.\]

Finally, observe that
\[\mathcal K_2(x-y)\times \dfrac{\xi(y)}{\kappa}=\dfrac{1}{16\pi^3\kappa^3}
\dfrac{1}{|\tilde x-\tilde y|^2}(\tilde x-\tilde y,0)\times
(\xi(y)-\xi(x)+\xi(x))\]
\[=\dfrac{1}{16\pi^3\kappa^3}(0,0,1)+\dfrac{1}{16\pi^3\kappa^3}\dfrac{1}{|\tilde
x-\tilde y|^2}(\tilde x-\tilde y,0)\times \xi(x).\]

Moreover, since $x\in U_R$ then $(1-\varphi_{\delta}(x-y))(1-\zeta_M(y))=
(1-\zeta_M(y))$. Thus,
\[|\mathcal J^2_n(t,x)|\leq \left|\dfrac{1}{16\pi^3\kappa^3}\int_{\Omega}
(1-\zeta_M(y))(\w_n(t,y)-\w(t,y))dy\right|\]
\[+C\int_{\Omega}(1-\zeta_M(y))\dfrac{1}{|\tilde x-\tilde y|}(|\w_n(t,y)|+
|\w(t,y)|)dy\]
\[\leq \left|\dfrac{1}{16\pi^3\kappa^3}\int_{\Omega}(1-\zeta_M(y))(\w_n(t,y)-
\w(t,y))dy\right|+\dfrac{C}{M}.\]

Observe that, as $\int_{\Omega}\w_n(t,y)dy=0=\int_{\Omega}\w(t,y)dy$ and
$\zeta_M\in L^{p'}$, we have
\[\lim_{n\rightarrow \infty}\int_{\Omega}(1-\zeta_M(y))(\w_n(t,y)-\w(t,y))dy=
-\lim_{n\rightarrow \infty}\int_{\Omega}\zeta_M(y)(\w_n(t,y)-\w(t,y))dy=0.\]

Thus, $\displaystyle \lim_{n\rightarrow\infty} \mathcal J^2_n(t,x)=\mathcal
O(M^{-1})$. Therefore, we find that
\[\lim_{n\rightarrow \infty}\int_{\Omega}(1-\varphi_{\delta}(x-y))\mathcal
K_2(x-y) \times \dfrac{\xi(y)}{\kappa}(\w_n(t,y)-\w(t,y))dy=\lim_{M\rightarrow
\infty} \lim_{n\rightarrow \infty}(\mathcal J^1_n(t,x)+\mathcal J^2_n(t,x))=0.\]

Furthermore, it follows from the
calculations above that
\[\left|\int_{\Omega}(1-\varphi_{\delta}(|x-y|))\mathcal K_2(x-y)
\times \dfrac{\xi(y)}{\kappa}(\w_n(t,y)-\w(t,y))dy\right|\leq |\mathcal
J^1_n(t,x)|+|\mathcal J^2_n(t,x)| \]\[\leq
C\left(\|f_x\|_{L^{\infty}}+\frac{1}{M}+1\right)\|\w_n-\w\|_{L^1}\leq
C_{\delta, M}.\]

Thus, by the Dominated Convergence theorem, we conclude that
$\lim_{n\rightarrow \infty}  \mathcal F^2_n=0$.

Finally, we end up with
\[\lim_{n\rightarrow \infty}\|u_n(t,\cdot)-u(t,\cdot)\|^2_{L^2(U_R)}=
\lim_{\delta\rightarrow 0}\lim_{n\rightarrow \infty}(\mathcal I^n_1+\mathcal
I^n_2)=0.\]

It is not hard to see that all the estimates used were uniform in $t$, so that
\[\lim_{n\rightarrow \infty}\|u_n-u\|_{L^{\infty}(0,\infty;
L^2(U_R))}=0.\]

Therefore, $u_n\rightarrow u$ in $L^{\infty}(0,\infty;L^2_{loc,per}(\Omega,
\mathbb R^3))$ hence follows the desired result.
\end{proof}

\begin{remark} We observe that, from the proof of the last theorem, it follows that
$u$ is a weak solution of the Euler equations in the velocity formulation.
Indeed, 
since each $u_n$ is a smooth solution of the Euler equations, we find that, for all test function $\Phi\in \mathcal
C_{c,per}^\infty([0,\infty)\times\Omega,\mathbb R^3)$ such that div $\Phi=0$, the following identity holds true:
\[\int_0^\infty\int_\Omega  \Phi_t(t,x)\cdot u_n(t,x)dxdt+\int_0^{\infty}
\int_{\Omega}[(u_n(t,x)\cdot \nabla)\Phi(t,x)]\cdot u_n(t,x)dxdt+\int_{\Omega}
\Phi(0,x)\cdot u_n(0,x)dx=0.\]
Now, since
$u_n\rightarrow u$ in $L^\infty(0,\infty;L_{loc,per}^2(\Omega, \mathbb R^3)$, we can pass to the limit, as $n\rightarrow \infty$, on the left-hand-side
above and obtain that $u$ is a weak solution of the Euler equations.
\end{remark}


\section{Second existence theorem: general case}\label{sec4}
In this section we will prove that the existence result can   also be obtained in the case where the initial vorticity does
not have vanishing integral. In order to do so, we must introduce another definition
of weak solution  of the vorticity equation, adapted to this case.

We can not use the Biot-Savart law in the case where the integral of the
vorticity is not zero. In order to deal with this issue we define a new operator
which plays the role of the Biot-Savart law. The idea is to decompose
the vorticity in two terms, one with zero integral and the other being the
vorticity associated with a helical velocity field which is a smooth, radially
symmetric, steady solution of the Euler equations. This construction is 
inspired by the radial-energy decomposition of a velocity field (see
\cite{Majda}). The velocity
field can then be recovered from the vorticity in the following simple way. We
apply the Biot-Savart law to the term which has zero integral and, for the other
term,
we already have the associated velocity field. With this idea in mind, let us
give the following definition:

\begin{definition} Let $p>4/3$. We define the operator
\begin{eqnarray}
\varXi:L_{per}^p(\Omega;\mathbb R)\cap L_{per}^1(\Omega;\mathbb R)
&\longrightarrow & L_{loc,per}^2(\Omega;\mathbb R)\\\nonumber
\w&\longmapsto &\varXi [\w]
\end{eqnarray}
as \[\varXi [\w](x):=\int_{\Omega}\mathcal K(x-y)\times\dfrac{ \xi(y)}{\kappa}
(\w(y)-\varphi(|\tilde y|))dy + \bar u(x),\] where $\varphi:
[0,\infty)\rightarrow \mathbb R$ is any smooth function with compact
support in $(0,\infty)$ such that
\[\int_0^\infty \varphi(r)rdr=\int_{\mathbb R^2}\w(\tilde x,0)d\tilde x,\]
and $\bar u$ is defined for all $x\in
\Omega$ by
\[\bar u(x)=\left(\dfrac{\tilde
x^{\perp}}{|\tilde
x|^2},\dfrac{1}{\kappa}\right)\int_0^{|\tilde x|}\varphi(r)rdr,\]
where $\tilde x^\perp=(-x_2,x_1)$.
\end{definition}

\begin{remark}
The operator $\varXi$ is well-defined. Indeed, since $\w\in
L_{per}^p(\Omega;\mathbb R)\cap L_{per}^1(\Omega;\mathbb
R)$ and $\varphi\in C_c^\infty((0,\infty))$ it is not hard to see
that  $\mathcal K\ast \xi(\cdot)/\kappa(\w(\cdot)-\varphi(|\tilde
\cdot|))\in L_{loc,per}^2(\Omega;\mathbb
R)$. Now, let us prove that the definition of $\varXi$ does not depend  on
the particular choice of $\varphi$. To do so, consider $\varphi, \zeta \in
\mathcal
C^{\infty}([0,\infty))$ with compact support in $(0,\infty)$ and such that
\[\int_0^{\infty}\varphi(r)rdr=\int_0^{\infty}\zeta(r)rdr=\int_{\mathbb
R^2}\w(\tilde x,0)d\tilde x.\]
Set \[\bar
u(x)=\left(\dfrac{\tilde
x^{\perp}}{|\tilde x|^2},\dfrac{1}{\kappa}\right)\int_0^{|\tilde
x|}\varphi(r)rdr \mbox{ and }\bar v(x)=\left(\dfrac{\tilde
x^{\perp}}{|\tilde x|^2},\dfrac{1}{\kappa}\right)\int_0^{|\tilde
x|}\zeta(r)rdr.\]
Observe that it is enough to prove that
\begin{equation}\label{inddef}
 \int_{\Omega}\mathcal
K(x-y)\times\dfrac{ \xi(y)}{\kappa} (\zeta(|\tilde y|)-\varphi(|\tilde y|))dy
+\bar u(x)-\bar v(x)=0.
\end{equation}
First, observe that $ \curl \bar u(y)= \xi(y)/\kappa\varphi(|\tilde
y|)$ and $ \curl \bar v(y)= \xi(y)/\kappa\zeta(|\tilde y|)$. Hence 
\eqref{inddef} is equivalent to
\begin{equation}\label{id01}
 \int_{\Omega}\mathcal K(x-y)\times \curl(\bar v(y)-\bar u(y))dy=\bar
v(x)-\bar u(x).
\end{equation}
Finally, observe that the vector field defined by $\Phi:=\curl (\bar v-\bar
u)$ is such that $\int_\Omega \Phi(x)dx=0$,
$\di \Phi=0$, $\Phi\in C^\infty_{c,per}(\Omega;\mathbb R^3)$ and $|\Phi(x)|$
goes to zero as $|\tilde x|\rightarrow \infty$; we can, therefore,  use Proposition
\ref{resolvebs} to obtain the identity \eqref{id01}.
\end{remark}

The definition of weak solution in the case where the vorticity
is not balanced follows naturally, and it is stated below:

\begin{definition}\label{w-g}Given $\w^0\in L_{c,per}^p(\Omega;\mathbb R)$, for
 some $p>4/3$, we say that $\w=\w(t,x)$ is a weak solution of
$\w_t+(u\cdot \nabla)\w=0$ if
\begin{itemize}
\item[(i)]$\w\in L^{\infty}(0,T; L_{per}^p(\Omega;\mathbb
R)\cap L_{per}^1(\Omega;\mathbb R))$;
\item [(ii)] the velocity field $u(t,\cdot)=\varXi [\w (t,\cdot)]$, for a.e.
$t\in [0,T]$;
\item [(iii)] for all test function $\psi \in \mathcal
C^{\infty}_c([0,T)\times\Omega;\mathbb R)$
\end{itemize}
\[\int_0^{T}\int_{\Omega} \psi_t(t,x) \w(t,x)dxdt+\int_0^{T}
\int_{\Omega}u(x,t)\cdot \nabla
\psi(x)\w(t,x)dxdt+\int_{\Omega} \psi(x,0) \w^0(x)dx=0.\]
\end{definition}

Next, we state and prove existence of a weak solution for the
vorticity equation without the integral constraint on the initial vorticity:

\begin{theorem}\label{teo2}
Given $\w^0\in L_{c,per}^p(\Omega;\mathbb R)$, for some $p>4/3$,
there exists a weak solution $\w=\w(t,x)$ of (\ref{vort}) in the sense of
Definition \ref{w-g}.
\end{theorem}
\begin{proof}
We have already proved the theorem in the case $\int \w^0dx=0$. Now, suppose
that $\int \w^0dx\neq 0$. Fix $\varphi \in \mathcal C^{\infty}([0,\infty))$,
with compact support in $(0,\infty)$, such that
\[\int_0^{\infty}\varphi(r)rdr=\int_{\mathbb R^2}\w^0(\tilde x,0)d\tilde x\]
and define  $\tilde \w^0=\w^0 -\phi$, where
$\phi=\phi(x)=\varphi(|\tilde x|)$.
We introduce an approximating sequence $\tilde \w^0_n$ given by
\[\tilde \w^0_n=\w^0_n-\phi,\]
where $\w^0_n$ is the sequence of mollifiers defined by $\w^0_n= \rho_n\ast
\w^0$. Observe that $\int_{\Omega}
\tilde \w^0_ndx=0$ for all $n$. Thus, for all $n$ we can define the velocity
field associated with $\w^0_n$ through the Biot-Savart law,
\[\tilde u_n^0=\mathcal K\ast \left(\tilde\w^0_n\dfrac{\xi}{\kappa}\right).\]
By Corollary \ref{corK} we have that, in particular, $\tilde
u_n^0\in\mathcal C^{\infty}(\Omega;\mathbb R^3)\cap L^2_{per}(\Omega;\mathbb
R^3)$ for all $n$. Using the standard energy estimates it is not
hard to prove that, for each $n\in\mathbb N$, the
system
\begin{eqnarray}\label{utilde}\left\{\begin{array}{ll}
\partial_t\tilde u+(\tilde u\cdot \nabla)\tilde
u+(\tilde u\cdot \nabla)\bar u+(\bar u\cdot \nabla)\tilde u+\nabla \tilde P=0\\
\mbox{div } \tilde u=0\\
\tilde u(0,x)=\tilde u_n^0(x)\\
|\tilde u|(x)\rightarrow 0 \mbox{ as } |\tilde x|\rightarrow \infty.
         \end{array}\right.
\end{eqnarray}
has a unique solution in $\mathcal C^{\infty}([0,\infty)\times\Omega;\mathbb
R^3)\cap L^\infty(0,\infty;L^2_{per}(\Omega;\mathbb R^3))$. Let $\tilde u_n\in
\mathcal C^{\infty}([0,\infty)\times\Omega;\mathbb R^3)\cap
L^\infty(0,\infty;L^2_{per}(\Omega;\mathbb R^3))$ be such solution.
It is easy to see that $\tilde u_n^0$ is helical and
has vanishing helical swirl and, therefore, $\tilde u_n$ is also
helical and has vanishing helical swirl, for all $n\in\mathbb N$. Now, set
\[u_n=\tilde u_n+\bar u\quad \mbox{ where }\quad \bar
u(x)=\left(\dfrac{\tilde x^{\perp}}{|\tilde
x|^2},\dfrac{1}{\kappa}\right)\int_0^{|\tilde
x|}\varphi(r)rdr.\]
We have that the velocity field $u_n$ satisfies the following properties:
\begin{itemize}
 \item [(i)] for each $n$, $u_n$ is a solution of the following problem
\begin{eqnarray}\label{utotal}\left\{\begin{array}{ll}
\partial_t u+(u\cdot \nabla)u+\nabla P=0\\
\mbox{div } u=0\\
u(0,x)=\tilde u_n^0(x)+\bar u(x)\\
|(u_1,u_2)|(x)\rightarrow 0 \mbox{ and } |u_3|(x)\rightarrow \beta \mbox{ as
} |\tilde x|\rightarrow \infty.
         \end{array}\right.\end{eqnarray}
where $\beta=(1/\kappa) |\int_0^\infty\varphi(r)rdr|$, $P=\tilde P+\bar P$ and
$\bar P$ is the pressure associated with the steady solution $\bar u$ of the
Euler equations;
 \item [(ii)] $u_n$ is helical and has vanishing helical swirl for all $n$;
 \item[(iii)]denote the third component of curl $u_n$ by
$\w_n$, then curl $u_n=\w_n \xi/\kappa$ for
all $n$;
\item [(iv)] $u_n(t,x)=\varXi \w_n(t,x)=\left(\mathcal
K\ast \left( (\w_n-\phi)\xi/\kappa\right)\right)(t,x)+\bar
u(x)$.
\end{itemize}

It is also clear that, for each $n\in \mathbb N$, $\w_n \in
C^\infty([0,\infty)\times \Omega)$ and
$\partial_t \w_n+(u_n\cdot\nabla)\w_n=0$. Therefore, for all $1\leq q\leq p$ we
have $\|\w_n(t)\|_{L^q(\Omega)}=\|\w_n(0)\|_{L^q(\Omega)}\leq
\|\w^0\|_{L^q}\leq C$, for all $n\in \mathbb N$ and $t\geq 0$. Then, there
exist a subsequence, which we still denote by $\{\w_n\}_n$, and a limit $\w\in
L^\infty(0,\infty;L_{per}^p(\Omega;\mathbb R)\cap L_{per}^1(\Omega;\mathbb R))$
such that $\w_n\stackrel{*}{\rightharpoonup}\w$ in
$L^\infty(0,\infty;L^p(\Omega;\mathbb R))$. Furthermore,
\[\int_0^{T}\int_{\Omega} \psi_t(t,x) \w_n(t,x)dxdt+\int_0^{T}
\int_{\Omega}u_n(t,x)\cdot \nabla \psi(x)
\w_n(t,x)dxdt+\int_{\Omega} \psi(x,0)\w^0_n(x)dx=0.\]

To show that $\w$ is a weak solution of vorticity equation in the
sense of Definition \ref{w-g}, we have to prove that we can pass to the limit on the left-hand-side of the identity above, as 
$n\rightarrow \infty$:
true
\begin{multline}\label{lasteq}
\lim_{n\rightarrow \infty}\left\{\int_0^{T}\int_{\Omega}
\psi_t(t,x) \w_n(t,x)dxdt+\int_0^{T}
\int_{\Omega}u_n(t,x)\cdot \nabla \psi(x)
\w_n(t,x)dxdt+\int_{\Omega} \psi(x,0)\w^0_n(x)dx\right\}=\\
=\int_0^{T}\int_{\Omega} \psi_t(t,x) \w(t,x)dxdt+\int_0^{T}
\int_{\Omega}u(t,x)\cdot \nabla \psi(x)
\w(t,x)dxdt+\int_{\Omega} \psi(x,0)\w^0(x)dx.
\end{multline}

The proof of relation \eqref{lasteq} follows along the same lines as the proof
of Theorem \ref{teores}, using the fact that $u_n=\varXi \w_n$.
\end{proof}

We conclude with a few final observations. First, the critical regularity for
existence of a weak solution in the 2D and axisymmetric case is $p=1$, see
\cite{delort90,vecchiwu,CI98}.  We expect helical flows to behave similarly, so
that $p> 4/3$ should
not be sharp, just a consequence of the limitations of the estimates available.
However, the
symmetrization idea that brings the existence result from $p>4/3$ to vortex
sheets in the 2D case, see
\cite{schochet},  has already been used in the proof of Theorem \ref{teores}, so
improvements have to come
from somewhere else. Some natural open problems are, therefore, proving
existence closer to vortex
sheet data, seeking an adaptation of Delort's axisymmetric result to helical
flows and formulating and
studying convergence of helical vortex methods. Another natural line of research
is to study the vanishing
viscosity limit, which is currently under investigation.

\medskip
\footnotesize
{\bf Acknowledgments:} Research of H. J. Nussenzveig Lopes is supported in part by CNPq grant \# 306331 / 2010-1, 
FAPERJ grant \# E-26/103.197/2012. Research of Milton C. Lopes Filho is supported in part by CNPq grant 
\# 303089 / 2010-5. Anne Bronzi's research is supported by Post-Doctoral grant \# 236994/2012-3.
This work was supported by FAPESP Thematic Project \# 2007/51490-7 and by the FAPESP grant 
\# 05/58136-9. The research presented here was part of Anne Bronzi's doctoral dissertation at the
mathematics graduate program of UNICAMP.

\bibliographystyle{plain}
\bibliography{references}

\begin{thebibliography}{10}

\bibitem{C05}
D.~Chae.
\newblock Remarks on the blow-up criterion of the three-dimensional {E}uler
  equations.
\newblock {\em Nonlinearity}, 18(3):1021--1029, 2005.

\bibitem{C08}
D.~Chae.
\newblock On the blow-up problem for the axisymmetric 3{D} {E}uler equations.
\newblock {\em Nonlinearity}, 21(9):2053--2060, 2008.

\bibitem{CI98}
D.~Chae and O.~Y. Imanuvilov.
\newblock Existence of axisymmetric weak solutions of the {$3$}-{D} {E}uler
  equations for near-vortex-sheet initial data.
\newblock {\em Electron. J. Differential Equations}, (26), 1998.

\bibitem{CK97}
D.~Chae and N.~Kim.
\newblock Axisymmetric weak solutions of the {$3$}-{D} {E}uler equations for
  incompressible fluid flows.
\newblock {\em Nonlinear Anal.}, 29(12):1393--1404, 1997.

\bibitem{DeSz1}
C.~DeLellis and L.~Sz{\'e}kelyhidi Jr.
\newblock On admissibility criteria for weak solutions of the {E}uler
  equations.
\newblock {\em Arch. Ration. Mech. Anal.}, 195(1):225--260, 2010.

\bibitem{delort90}
J.-M. Delort.
\newblock Existence de nappes de tourbillon en dimension deux.
\newblock {\em J. Amer. Math. Soc.}, 4(3):553--586, 1991.

\bibitem{D92}
J.-M. Delort.
\newblock Une remarque sur le probl\`eme des nappes de tourbillon
  axisym\'etriques sur {$\mathbb R^3$}.
\newblock {\em J. Funct. Anal.}, 108(2):274--295, 1992.

\bibitem{Dutrifoy}
A.~Dutrifoy.
\newblock Existence globale en temps de solutions h\'elico\"\i dales des
  \'equations d'{E}uler.
\newblock {\em C. R. Acad. Sci. Paris S\'er. I Math.}, 329(7):653--656, 1999.

\bibitem{Titi}
B.~Ettinger and E.~Titi.
\newblock Global existence and uniqueness of weak solutions of
  three-dimensional {E}uler equations with helical symmetry in the absence of
  vorticity stretching.
\newblock {\em SIAM J. Math. Anal.}, 41(1):269--96, 2009.

\bibitem{evans}
L.~C. Evans.
\newblock {\em Partial differential equations}, volume~19 of {\em Graduate
  Studies in Mathematics}.
\newblock American Mathematical Society, Providence, RI, 1998.

\bibitem{Folland1}
G.~B. Folland.
\newblock {\em Fourier analysis and its applications}.
\newblock The Wadsworth \& Brooks/Cole Mathematics Series. Wadsworth \&
  Brooks/Cole Advanced Books \& Software, Pacific Grove, CA, 1992.

\bibitem{GZ07}
S.~Gang and X.~Zhu.
\newblock Axisymmetric solutions to the {3D} {E}uler equations.
\newblock {\em Nonlinear Anal.}, 66(9):1938--1948, 2007.

\bibitem{JX06}
Q.~Jiu and Z.~Xin.
\newblock On strong convergence to {3D} axisymmetric vortex sheets.
\newblock {\em J. Differential Equations}, 233(1):33--50, 2006.

\bibitem{Kato}
T.~Kato.
\newblock Nonstationay flows of viscous and ideal fluids in {$\mathbb R^3$}.
\newblock {\em J. Funct. Anal.}, 9:296--305, 1972.

\bibitem{MOS}
W.~Magnus, F.~Oberhettinger, and R.~P. Soni.
\newblock {\em Formulas and Theorems for the Special Functions of Mathematical
  Physics}.
\newblock Springer-Verlag New York, Inc., New York, third edition, 1966.

\bibitem{MTL90}
A.~Mahalov, E.~S. Titi, and S.~Leibovich.
\newblock Invariant helical subspaces for the {N}avier-{S}tokes equations.
\newblock {\em Arch. Rational Mech. Anal.}, 112(3):193--222, 1990.

\bibitem{M86}
A.~Majda.
\newblock Vorticity and the mathematical theory of incompressible fluid flow.
\newblock {\em Comm. Pure Appl. Math.}, 39(S):S187--S220, 1986.

\bibitem{Majda}
A.~Majda and A.~Bertozzi.
\newblock {\em Vorticity and incompressible flow}, volume~27 of {\em Cambridge
  Texts in Applied Mathematics}.
\newblock Cambridge University Press, Cambridge, 2002.

\bibitem{SR94}
X.~Saint Raymond.
\newblock Remarks on axisymmetric solutions of the incompressible {E}uler
  system.
\newblock {\em Comm. Partial Differential Equations}, 19(1--2):321--334, 1994.

\bibitem{schochet}
S.~Schochet.
\newblock The weak vorticity formulation of the {$2$}-{D} {E}uler equations and
  concentration-cancellation.
\newblock {\em Comm. Partial Differential Equations}, 20(5--6):1077--1104,
  1995.

\bibitem{vecchiwu}
I.~Vecchi and S.~Wu.
\newblock On {$L\sp 1$}-vorticity for {$2$}-{D} incompressible flow.
\newblock {\em Manuscripta Math.}, 78(4):403--412, 1993.

\bibitem{W2012}
E.~Wiedemann.
\newblock Existence of weak solutions for the incompressible {E}uler equations.
\newblock {\em Ann. Inst. H. Poincar\'e Anal. Non Lin\'eaire}, 28(5):727--730,
  2011.

\end{thebibliography}
\end{document}